\documentclass[reqno]{amsart}
\usepackage{mathtools}
\usepackage{microtype}
\usepackage[mathscr]{eucal}
\usepackage{stmaryrd}
\usepackage{slashed}
\usepackage{csquotes}
\usepackage[normalem]{ulem}

\DeclareMathOperator{\curl}{curl}
\DeclareMathOperator{\tr}{tr}
\DeclareMathOperator{\Span}{span}
\DeclareMathOperator{\dom}{Dom}
\DeclareMathOperator{\re}{Re}
\DeclareMathOperator{\im}{Im}

\DeclareMathOperator{\Div}{div}
\DeclareMathOperator{\Ker}{Ker}
\DeclareMathOperator{\ind}{ind}

\DeclareMathOperator{\Id}{Id}

\DeclarePairedDelimiter\IP{\langle}{\rangle}
\DeclarePairedDelimiter\abs{\lvert}{\rvert}
\DeclarePairedDelimiter\norm{\lVert}{\rVert}
\DeclarePairedDelimiter\AC{\lceil}{\rceil}

\newcommand{\C}{\mathbb{C}}
\newcommand{\Z}{\mathbb{Z}}
\newcommand{\R}{\mathbb{R}}
\newcommand{\N}{\mathbb{N}}
\newcommand{\Hb}{\mathscr{H}^{\Omega}}
\newcommand{\Hbc}{\mathscr{H}^{\Omega,c}}
\newcommand{\Hcr}{\mathscr{H}_{\Ab,g}^{\mathcal O}}
\newcommand{\Hcc}{\mathscr{H}_{\Ab,g}^{\bar{\mathcal O}}}
\newcommand{\Hm}{\mathscr{H}_m}
\newcommand{\Hbm}{\mathscr{H}^{\Omega,m}}

\newcommand{\Ab}{\mathbf{A}}

\newcommand{\Dd}{\mathscr{D}}

\newcommand{\dd}{\mathop{}\!{\mathrm d}}

\newcommand{\ii}{\mathrm i}
\newcommand{\ee}{\mathrm e}
\newcommand{\jb}{\mathbf j}
\newcommand{\eb}{\mathbf e}
\newcommand{\n}{\mathrm n}
\newcommand{\dsum}{\oplus_{j=0}^d}

\newtheorem{theorem}{Theorem}[section]
\newtheorem{lemma}[theorem]{Lemma}
\newtheorem{proposition}[theorem]{Proposition}

\newtheorem{corollary}[theorem]{Corollary}
\theoremstyle{remark}
\newtheorem{remark}[theorem]{Remark}

\numberwithin{equation}{section}

\usepackage[colorlinks=true,linkcolor=red,citecolor=black]{hyperref}

\author{S. Fournais}
\author{R. L. Frank}
\author{M. Goffeng}
\author{A. Kachmar}
\author{M. Sundqvist}

\address[S. Fournais]{Department of Mathematics, University of Copenhagen, Universitets\-parken 5, DK-2100 Copenhagen \O, Denmark}
\email{sofo@math.ku.dk}

\address[R. L. Frank]{Department of Mathematics, Ludwig-Maximilans Universit\"at M\"unchen, Germany, and Munich Center for Quantum Science and Technology, Germany, and Department of Mathematics, Caltech, USA}
\email{r.frank@lmu.de}

\address[M. Goffeng]{Department of Mathematics, Lund University, Sweden}
\email{magnus.goffeng@math.lth.se}

\address[A. Kachmar]{Department of Mathematics, Lebanese University, Nabatiyeh, Lebanon.}
\email[A. Kachmar]{ayman.kashmar@gmail.com}

\address[M. Sundqvist]{Department of Mathematics, Lund University, Sweden}
\email{mikael.persson\_sundqvist@math.lth.se}

\title[Negative Eigenvalues for the Robin Pauli operator]{Counting Negative Eigenvalues for\\ the magnetic Pauli operator}


\begin{document}
\begin{abstract}
We study the Pauli operator in a two-dimensional, connected domain with Neumann or Robin boundary  condition. We prove a sharp lower bound on the number of negative eigenvalues reminiscent of the Aharonov--Casher formula. We apply this lower bound to obtain a new formula on the number of eigenvalues of the magnetic Neumann Laplacian in the semi-classical limit. Our approach relies on reduction to a boundary Dirac operator. We analyze this boundary operator in two different ways.  The first approach uses Atiyah--Patodi--Singer index theory. The second approach relies on  a conservation law for the Benjamin--Ono equation. 
\end{abstract}

\maketitle

\section{Introduction}

\subsection{Motivation and background}

The statics and dynamics of a quantum mechanical particle are described by the Schr\"odinger equation. Of particular interest are bound states, which correspond to eigenvalues of the underlying Schr\"odinger operator, and in many applications, both practical and theoretical, one is interested in computing, or estimating, their number.
	
The main goal of this paper is to obtain such an estimate on the number of bound states in the setting of a spin-1/2 charged particle moving in a bounded subset of the plane in the presence of a magnetic field. At the same time, we will show that our estimate is best possible in a variety of asymptotic and non-asymptotic cases.

Our main result is related to the celebrated Aharonov--Casher theorem, which concerns the case of a particle moving in the whole plane. As we will recall momentarily, this theorem provides an identity for the number of zero eigenvalues in terms of the normalized flux of the magnetic field. A special case of our findings is that, if the particle is restricted to a bounded, connected subset by means of Neumann boundary conditions, then the Aharonov--Casher equality turns into an inequality for the number of resulting negative eigenvalues. This is remarkable since, in general, there is no monotonicity under Neumann boundary conditions.

As a consequence we obtain an isochoric inequality: among all simply connected sets of a given area, the disc has the least number of negative eigenvalues in the presence of a homogeneous field.

To prove our result we explore deep links between Pauli operators,  Atiyah--Patodi--Singer index theory and a conservation law for the Benjamin--Ono equation.  Such relations are established via tools in complex analysis and pseudodifferential calculus.

In the remainder of this introduction, we recall some background on Pauli operators and the Aharonov--Casher theorem before defining the operators of interest to us. Then we will state our main results and briefly sketch some ideas of the proof.

\subsubsection{The Aharonov--Casher theorem for Pauli operators}\label{subsec:AC}

The Pauli operator in the plane is the Hamiltonian describing a spin-1/2 charged particle subject to a magnetic field; it acts on $\C^2$-valued functions and is given by
\begin{equation}
\label{eq:Pauli}
    \begin{pmatrix}
        (-\ii\nabla-\Ab)^2 - B&0\\
        0&(-\ii\nabla-\Ab)^2 + B 
    \end{pmatrix}\equiv[\sigma\cdot(-\ii\nabla-\Ab)]^2 , 
    \end{equation}
where \(\Ab=(A_1,A_2):\R^2\to\R^2\) is a vector field, and \(\sigma=(\sigma_1,\sigma_2)\) denotes the vector of Pauli matrices,
\[
\sigma_1 = \begin{pmatrix}
	0 & 1 \\ 1 & 0
\end{pmatrix},
\qquad
\sigma_2 = \begin{pmatrix}
	0 & -\ii \\ \ii & 0
\end{pmatrix}.
\]
The magnetic field is the function 
\[ B=\curl\Ab = \partial_1 A_2 - \partial_2 A_1 \,. \]
Under suitable assumptions on $\Ab$, the Pauli operator is a self-adjoint, non-negative operator in the Hilbert space $L^2(\R^2;\C^2)$. The Aharonov--Casher theorem~\cite{AC} states that it can have zero as an eigenvalue.  In fact,  the multiplicity of the zero eigenvalue is equal to
\[
\max\{\lceil\lvert \Phi \rvert\rceil-1,0\},
\]
where \(\Phi\) denotes the normalized flux 
\[
\Phi = \frac{1}{2\pi} \int_{\R^2} B(x) \dd x
\]
of the magnetic field and where we use the notation
\begin{equation}\label{eq:def-K}
	\AC{x}
	\coloneqq
	\min\{ z \in {\mathbb Z} ~:~ z \geq x\}.
\end{equation}
The Aharonov--Casher result has later been extended to also include singular magnetic fields and to magnetic fields with infinite flux~\cite{ervo,pe1,pe2,rosh1,rosh2}. The proofs rely on the fact that the elements of the kernel must be (anti)holomorphic after multiplication with a certain weight.  Recently, the number of zero modes for the Dirac operator in certain two-dimensional manifold were calculated in~\cite{Fi}.

\subsubsection{Negative eigenvalues for Pauli operators on domains}\label{subsec:Pauli}

In this work, we will be interested in the Pauli operator on a bounded domain \(\Omega\). 

One natural choice of boundary conditions is that of Dirichlet boundary condition. This was studied in~\cite{elton} and a flux effect in the eigenvalue counting-function was proved, in a certain semi-classical (or strong magnetic field) regime. The result is similar to the Aharonov--Casher theorem, although the imposed Dirichlet boundary condition splits the eigenvalues and shifts them upwards~\cite{BLRS}.  Such a semi-classical result and its proof relate to the seminal work~\cite{LSY} (see also~\cite{PhysRevB.51.10646}).

Another natural choice is that of Neumann boundary conditions, and this is what we are interested in here in this work.  The importance of Neumann boundary conditions is found in its physical relevance, in for instance superconductivity~\cite{FH-b}. Such a Pauli operator is not the square of a Dirac operator, which opens up the possibility of negative eigenvalues.

\subsubsection{A surprising observation}\label{supplies}

Our work started when we noticed from the figure in~\cite{sage} that\footnote{They consider the magnetic Schrödinger operator and not the Pauli operator, but for uniform magnetic fields they differ by a shift.} in the case of the disc in a uniform magnetic field of strength \(\beta > 0\) (i.e., \(B(x)=\beta\)), and with Neumann boundary conditions, it is possible to count the negative eigenvalues exactly. We have reproduced a version of the graph\footnote{We used Wolfram Mathematica to do the calculations;  the eigenfunctions are Whittaker functions.} of the eigenvalues as functions of $\beta>0$ from~\cite{sage} in Figure~\ref{fig:disk}, modified for the Pauli operator.

\begin{figure}[htb]\label{fig:disk}
	\includegraphics{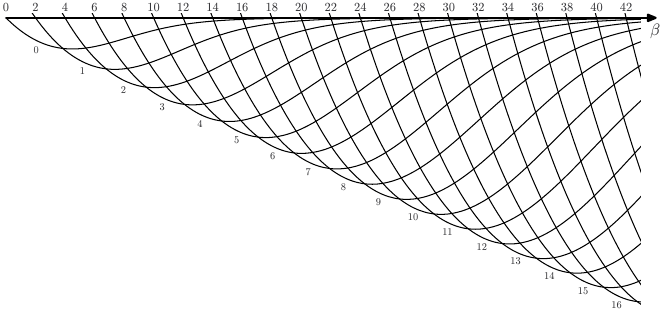}
	\caption{Eigenvalues in the case of the unit disc, with a uniform magnetic field of strength \(\beta\), where \(\beta > 0\) is a parameter, as a function of \(\beta\). The flux of the magnetic field is given by \(\Phi = \beta/2\). As \(\beta\) increases, new negative eigenvalue occurs each time \(\Phi\) passes an integer.}
\end{figure}

Working in polar coordinates,  the variables separate, and one is led to study a family of radial operators, indexed by the angular momentum \(m\), and the parameter \(\beta\). It turns out that it suffices to consider \(m\geq 0\), and that the operator corresponding to angular momentum \(m\) has precisely one negative eigenvalue if \(\beta/2 > m\geq 0\), i.e., if \(\Phi > m\), and no negative eigenvalue otherwise. Summing up, the number of negative eigenvalues equals \(\AC{\Phi}\).

\subsubsection{Our aims}

We will generalize the observation from Subsection~\ref{supplies}.  First, for radial, nonnegative, nonincreasing (with respect to the radial variable) magnetic fields in the disc, subject to a Neumann boundary condition,  we show that the number of negative eigenvalues again equals \(\AC{\Phi}\).  For a general smooth domain, and with Neumann boundary conditions, we prove a lower bound in terms of the flux.

Our methods allow us to handle more generally the case of Robin boundary conditions  and this will yield an effective flux-like term coming from the Robin data.  See Theorem~\ref{thm:main} below for a precise statement. We emphasize that, even in the case where the original problem has the more standard Neumann boundary condition, our proof passes through operators with Robin conditions in intermediate steps, so it is natural to consider the problem in this generality.

Our second objective concerns the asymptotics of the number of negative eigenvalues in the semiclassical limit. The Schr\"{o}dinger (and also Pauli) operator in a constant magnetic field on the entire plane has spectrum consisting of infinitely degenerate eigenvalues exactly at the product of the magnetic field strength $\beta$ with the odd (even) non-negative integer numbers. The density of states in these eigenspaces is given by $\frac{\beta}{2\pi}$ (times an integer factor in the Pauli case to account for the varying spin degeneracy). We recognize this as the flux per unit area. Many of the results on magnetic spectral asymptotics can be seen as based on this fact and investigating how it responds to different perturbations (non-constant field, external electric field, boundary conditions)~\cite{cdv,LSY,PhysRevB.51.10646,truc,FK,elton}.
However, these works leave open the asymptotics of the counting function exactly at a Landau level. This is clearly a more difficult case, since the high degeneracy at this energy implies a discontinuity in the counting function. Our Theorem~\ref{thm:main*} solves this problem in the most clear-cut case of a constant magnetic field and a Neumann boundary condition. The main new input to the proof of this theorem is the lower bound to the counting function from Theorem~\ref{thm:main}.

\subsection{The main result}

We will now introduce the quantities of interest and state our main result.

We consider a bounded, open set \(\Omega\subset\R^2\) such that \(\Gamma \coloneqq  \partial\Omega\) is a finite union of \(d+1\) $C^\infty$-smooth curves \(\Gamma_0,\cdots,\Gamma_d\). It is no loss of generality to assume that $\Omega$ is connected. Note that if \(\Omega\) is simply connected,  then \(d=0\). 

We assume that $B$, the magnetic field, is a smooth, real-valued function on $\overline\Omega$. The magnetic vector potential is a vector field $\mathbf{A}$, assumed smooth in $\overline\Omega$, that satisfies $\curl \Ab =B$. Note that if $\Omega$ is simply connected, then $\Ab$ is unique up to the addition of a gradient field.

Our main results will be stated in terms of the simpler magnetic Schr\"{o}dinger operator \((-\ii\nabla - \mathbf{A})^2 - B\) in \(L^2(\Omega)\), which can be considered as one of the components of the Pauli operator~\eqref{eq:Pauli}. Information about the other component of the Pauli operator can be retrieved from charge conjugation symmetry (namely, complex conjugation and flipping the vector components), which implements an anti-unitary equivalence between the Pauli operator with magnetic field $B$ and $-B$.

Given a function \(g\in C^\infty(\Gamma,\R)\) we impose a Robin\footnote{If \(\Ab\) is tangential to the boundary (as in~\eqref{eq:curl-A=B}),  then the magnetic Robin boundary condition \(\nu\cdot(-\ii\nabla - \mathbf{A})u+\ii gu=0\) coincides with the non-magnetic Robin condition \(\nu\cdot\nabla u=gu\).} boundary condition on the boundary \(\Gamma\), 
\begin{equation}
\label{eq:bc-Robin}
    \nu\cdot(\nabla-\ii\Ab) u|_{\partial\Omega}
    =
    g u|_{\partial\Omega}
    \quad 
    (u \in \dom(\Hb_{\Ab,g})).
\end{equation}
Here $\nu$ denotes the unit normal vector field of $\Gamma$, pointing towards the interior of $\Omega$. The operator 
\begin{equation*}
	\Hb_{\Ab,g} 
	= 
	(-\ii\nabla - \mathbf{A})^2 - B
\end{equation*}
is the self-adjoint operator in \(L^2(\Omega)\), associated with the closed, semi-bounded quadratic form
\begin{equation*}
    u\in H^1(\Omega)\mapsto q_{\Ab,g} (u)
    \coloneqq 
    \int_\Omega\big( \abs{(-\ii\nabla - \Ab)u}^2 - B\abs{u}^2\big)\dd x+\int_\Gamma g|u|^2\dd s(x).
\end{equation*}
Let us introduce the magnetic flux on \(\Omega\),
\begin{equation*}
    \Phi 
    \coloneqq
    \frac1{2\pi}\int_\Omega B(x)\dd x
    =
    \frac1{2\pi}\int_\Gamma\tau\cdot\Ab \dd s(x),
\end{equation*}
where $\tau$ is the unit tangent vector to $\Gamma$, oriented so that \((\tau,\nu)\) is a direct frame\footnote{In particular, if \(\Omega\) is simply connected, $\tau$ turns counterclockwise.}. Furthermore, we introduce on each connected component of \(\Gamma\)
\[
    \Phi_j
    \coloneqq
    \frac1{2\pi}\int_{\Gamma_j}\tau\cdot\Ab \dd s(x)
    \quad (0\leq j\leq d).
\]
We also introduce effective flux-like terms from the Robin function \(g\),
\begin{equation}\label{eq:def-Phi-g}
    \Phi_g 
    \coloneqq 
    \frac1{2\pi}\int_\Gamma g \dd s(x),\quad 
    \Phi_{g,j}
    \coloneqq \frac1{2\pi}\int_{\Gamma_j} g \dd s(x)\quad 
    (0\leq j\leq d).
\end{equation}
Observe that the magnetic flux on \(\Omega\) satisfies \(\Phi=\Phi_{\tau\cdot\Ab}\).

Given a self-adjoint operator \(T\) and a real number \(\lambda\), we denote by 
\[
    N(T,\lambda) 
    =
    \tr \big(\mathbf 1_{(-\infty,\lambda)}(T) \big),
    \quad
    N(T,\lambda_+) 
    =
    \tr \big(\mathbf 1_{(-\infty,\lambda]}(T)\big), 
\]
the number of eigenvalues of \(T\) that are less than (or equal) \(\lambda\).

The following is our main result.

\begin{theorem}
\label{thm:main}
Assume that \(\Omega\) is  bounded,  connected and smooth with \(d+1\) boundary components and that the magnetic field $B$ is smooth in $\overline{\Omega}$.  Let \(\Ab\) be a smooth solution to $\curl \Ab =B$. Then the number of negative eigenvalues of \(\Hb_{\Ab,g}\) satisfies
\[  
    N(\Hb_{\Ab,g},0)
    \geq     -d+ \sum_{j=0}^d \AC{\Phi_j-\Phi_{g,j}}.
\]
In particular,  if \(\Omega\) is simply connected, then
\[  
    N(\Hb_{\Ab,g},0)
    \geq     
    \AC{\Phi-\Phi_g}.
\]
Furthermore, if \(\Omega=D(0,R)\) is a disc of radius \(R\), if \(g=0\), and if the magnetic field \(B\geq 0\) is radial and radially nonincreasing, then
\[  
    N(\Hb_{\Ab,0},0)
    =
    \AC{\Phi}.
\]
\end{theorem}

The result for the disc shows that the lower bound in Theorem~\ref{thm:main} is rather optimal in the case of the Neumann boundary condition (\(g=0\)). In the simply connected case,  we note that the conclusion of Theorem~\ref{thm:main} is non-empty only when \(\Phi-\Phi_g> 0\), which holds if for instance $B>0$ and we impose Neumann boundary conditions.  The reader should be wary of the fact that when $\Omega$ is not simply connected,  the lower bound in Theorem~\ref{thm:main} depends on $\Ab$ and not just on $B$. For $\Omega$ simply connected, both $N(\Hb_{\Ab,g},0)$ and the lower bound in Theorem~\ref{thm:main} only depend on $B$.

A consequence of Theorem~\ref{thm:main} is  the following magnetic isochoric inequality.  That this was worth recording separately was noted by V.\ Lotoreichik.

\begin{corollary}\label{corol.is-in}
Assume that we have a Neumann boundary condition,  \(g=0\),  and that \(D(0,R)\) is the disc with same area as \(\Omega\).  If \(\Omega\) is simply connected and the magnetic field \(B>0\) is constant,   then 
\[
    N(\Hb_{B,0},0)
    \geq
    N(\mathscr H_{B,0}^{D(0,R)},0),
\]
where $N(\Hb_{B,0},0)$ denotes the number of negative eigenvalues of $\Hb_{\Ab,0}$ for any smooth solution $\Ab$ to $B=\curl\Ab$.
\end{corollary}

\begin{remark}
The smoothness condition $B\in C^\infty(\overline{\Omega},\R)$ in Theorem~\ref{thm:main} can readily be relaxed to the condition $B \in L^p(\Omega,\R)$, $ p>2$. We restrict to smooth fields for ease of presentation but we return to discuss the subtleties arising for $B \in L^p(\Omega,\R)$ below in Remark~\ref{rem:lpb}. The only reason for the smoothness assumption on $\Omega$ is to allow us to invoke Atiyah--Patodi--Singer's index theorem;   for further discussion see Remark~\ref{rem:lowome}.
\end{remark}

We will provide two different proofs of the lower bound on the number of negative eigenvalues in Theorem~\ref{thm:main}. The key observation underlying the first proof is that the number of negative eigenvalues is bounded from below by the Fredholm index of the Pauli--Dirac operator equipped with Atiyah--Patodi--Singer boundary conditions. Our argument shows that the number of negative eigenvalues of the Robin realization of a square of a twisted Dirac operator on any compact Riemannian manifold with boundary is bounded from below by an Atiyah--Patodi--Singer index, and in dimension $2$ the index can be computed in explicit geometric and magnetic invariants. This approach does not require that our domain is simply connected.

The second proof exhibits a connection with a conservation law for the Benjamin--Ono equation on the real line and uses the Riemann mapping theorem (so it does require that our domain is simply connected). This conservation law has the form of a trace formula that relates spectral data of a certain linear operator to its coefficient.  

We believe that the two proofs are complementary
and demonstrate interesting links between index theory and the trace formula. 

\subsection{Application: The semi-classical magnetic Neumann Laplacian}  With Neumann boundary condition (\(g=0\)), with uniform magnetic field \(B\equiv 1\), and with a semi-classical parameter \(h\ll1\), we obtain interesting results on the semi-classical, magnetic Neumann Laplacian
\[
    \mathcal L_h
    \coloneqq
    (-\ii h\nabla-\Ab)^2.
\]

\begin{theorem}\label{thm:main*}
Assume that \(\Omega\) is bounded,  connected and smooth with \(d+1\) boundary components and that the magnetic field \(B\equiv 1\).  Let \(\Ab\) be a smooth solution to $\curl \Ab =1$.   Then the number of eigenvalues of \(\mathcal L_{h}\)  strictly below \(h\) satisfies
\begin{equation*}
    N(\mathcal L_{h},h)
    \underset{h\to0}{\sim}
    \frac{|\Omega|}{2\pi}h^{-1}.
\end{equation*}
\end{theorem}

\begin{proof}
The upper bound 
\[
    N(\mathcal L_{h},h)
    \leq
    \frac{|\Omega|}{2\pi}h^{-1} +o(h^{-1})
\]
in Theorem~\ref{thm:main*} is well known and can be obtained  from~\cite[Thm.~1.4]{FK} by a standard argument of differentiation  with respect to the spectral parameter~\cite[Thm.~III.5]{LS} (see~\cite[Sec.~8.3]{N} for more details). The lower bound is a direct consequence of Theorem~\ref{thm:main}\footnote{We start with \(\Hb_{h^{-1}B,0}\) with \(B\equiv 1\) and find that \(\mathcal L_h - h= h^{2}\Hb_{h^{-1}B,0}\).}.
\end{proof}

The asymptotics in Theorem~\ref{thm:main*} is consistent with a previous result for the square~\cite[Thm.~1.3]{KN}. 

\begin{remark}
In~\cite{Fr}, the second listed author proved that 
\[
    N(\mathcal L_{h},\lambda h)
    \underset{h\to0}{\sim}
    c(\lambda)|\Gamma| h^{-1/2}
\]
where \(\lambda\in(0,1)\) is a fixed constant and \(c(\lambda)\) is an explicit constant depending on \(\lambda\) in such a way that \(c(\lambda)\to+\infty\) when \(\lambda\to 1^-\). All the involved eigenfunctions are localized close to the boundary $\Gamma$. Meanwhile, for $\lambda \in (1,3)$ it is fairly standard to show that 
\[
    N(\mathcal L_{h},\lambda h)
    \underset{h\to0}{\sim}
    \frac{|\Omega|}{2\pi}h^{-1}.
\]
In this parameter region, the density of eigenfunctions is evenly distributed over the entire domain $\Omega$. However, it is not clear from standard methods what the correct value, or even the correct exponent of $h$, should be exactly at $\lambda=1$, corresponding to the first Landau level. Our Theorem~\ref{thm:main*} solves this problem with the conclusion that there is bulk behavior even at the limiting point $\lambda=1$.
\end{remark}

\begin{corollary}\label{main:cor*}
Let the assumptions be as in Theorem~\ref{thm:main*}.  Then the sum of eigenvalues of \(\mathcal L_{h}\)  strictly below \(h\) satisfies
\begin{equation}\label{eq:semiclasssum}
    \sum_{e_j < h} e_j(\mathcal L_h)
    \underset{h\to0}{\to}
    \frac{|\Omega|}{2\pi}.
\end{equation}
\end{corollary}

\begin{proof}
In~\cite{FK},  the sum of eigenvalues strictly below \(h\) were asymptotically computed as
\[
    \sum_{e_j < h} \big(h-e_j(\mathcal L_h)\big)
    \underset{h\to0}{\sim}
    C_1\frac{|\Gamma|}{2\pi}h^{1/2}.
\]
Here \(C_1\in (0,\infty)\) is the universal constant defined as 
\[
    C_1:=\int_0^\infty(1-\mu_1(\xi))\dd \xi,
\]
where $\mu_1(\xi)$ denotes the lowest eigenvalue of the de Gennes operator $-\partial_t^2+(t-\xi)^2$ on $L^2(\R_+)$ with Neumann conditions.
Combined with Theorem~\ref{thm:main*}, we deduce~\eqref{eq:semiclasssum}.
\end{proof}

In the case of the disc, we have seen in Theorem~\ref{thm:main} that \(N(\mathcal L_{h},h)=\frac{|\Omega|}{2\pi}h^{-1}+\mathcal O(1)\). Hence,  when \(\Omega=D(0,R)\), the statement of Corollary~\ref{main:cor*} can be refined to a two-term expansion of the sum of eigenvalues below \(h\),
\[
    \sum_{e_j < h} e_j(\mathcal L_h)
    \underset{h\to0}{=}
    \frac{|\Omega|}{2\pi}- C_1\frac{|\Gamma|}{2\pi} h^{1/2}+o(h^{1/2})=\frac{R^2}{2}-C_1Rh^{1/2}+o(h^{1/2}).
\]
To put this into perspective, we note that a two-term expansion for the number of eigenvalues below $\lambda h$ with $\lambda\neq 1,3,5,\ldots$ was proved in~\cite{cffh} for general domains in the case of Dirichlet boundary conditions. A similar argument should be possible in the case of Neumann boundary conditions. However, we do not see how to use these methods to obtain two-term asymptotics at $\lambda h$ with $\lambda = 1,3,5,\ldots$.  The above results address this for $\lambda=1$ in the case of the disc. 

\subsection{Fixing the gauge}

Given a magnetic field \(B\in C^{\infty}(\overline{\Omega};\R)\), we consider the unique magnetic scalar potential \(\phi\in C^\infty(\overline{\Omega})\) satisfying
\begin{equation}\label{eq:SuperPot}
	\begin{cases}
		\Delta\phi = B & \text{ in }\Omega,\\
		\phi = 0       & \text{ on }\Gamma.
	\end{cases}
\end{equation}
In terms of the magnetic scalar potential, the vector potential 
\begin{equation}
\label{eq:def-A}
	\Ab(x) = (-\partial_2\phi,\partial_1\phi)
\end{equation}
is a solution to the boundary value problem:
\begin{equation}\label{eq:curl-A=B}
    \curl\Ab = B,\qquad
    \Div\Ab = 0, \qquad
    \quad\text{and}\qquad
    \nu\cdot\Ab=0 \text{ on }\Gamma \,.
\end{equation}
If $\Omega$ is simply connected,~\eqref{eq:curl-A=B} defines $\Ab$ uniquely,  and any magnetic potential with magnetic field \(B\) is gauge equivalent to \(\Ab\).  
If $\Omega$ is not simply connected,   however,  there are magnetic potentials generating the magnetic field  \(B\) that are not gauge equivalent to \(\Ab\).

\begin{remark}
\label{generapotential}
Theorem~\ref{thm:main} is stated for a general magnetic potential,  but we will throughout the proof make the simplifying assumption that $\Ab$ is defined from~\eqref{eq:def-A}.  With the exception of Section~\ref{sec:index}, this assumption is purely cosmetic. A less conceptual approach to the results of Section~\ref{sec:index} extending to general $\Ab$ is discussed in Remark~\ref{rem:generalmagn}. 
\end{remark}

\subsection{Overview of the paper}

The paper is organized as follows. In Section~\ref{discandradial} we prove the second statement of Theorem~\ref{thm:main} concerning the disc by means of an explicit diagonalization of the operator by separation of variables.

We proceed in Section~\ref{sec:useful} with providing the key ingredient in the proof of Theorem~\ref{thm:main}: the integration by parts formula given in Proposition~\ref{prop:identity}. This formula links the spectrum of $\Hb_{\Ab,g}$ to a first-order elliptic operator on the boundary. Although the analogous formula is well-known in the case of Dirichlet boundary conditions where the boundary term vanishes, we believe the formula in the setting of Robin boundary conditions to be new. Both our proofs of the first part of Theorem~\ref{thm:main} rely on the formula in Proposition~\ref{prop:identity}. 

In the first proof of Theorem~\ref{thm:main}, the lower bound on $N(\Hb_{\Ab,g},0)$ is obtained from restricting the quadratic form of $\Hb_{\Ab,g}$ to (a fixed factor times) holomorphic functions. The details of this first lower bound are given in Sections~\ref{sec:Dirac} and~\ref{sec:eginnad}. The integration by parts formula of Proposition~\ref{prop:identity} provides the crucial link to Atiyah--Patodi--Singer index theory that we use to finalise the first proof of Theorem~\ref{thm:main} in Section~\ref{sec:index}. Since we are applying Atiyah--Patodi--Singer index theory to a twisted $\bar{\partial}$-operator in the complex plane, we can keep the presentation at a pedestrian level and apply the index theorem~\cite{APS} as a black box. The explicit nature of the Atiyah--Patodi--Singer index in Section~\ref{sec:index} will likely not suprise the experts in the field of index theory and spin geometry~\cite{boosss,ggrubb,lawmich}, but we nevertheless hope the direct methods can be appreciated.

The second proof of Theorem~\ref{thm:main} is given in Section~\ref{sec:traceproof} and relies on a conservation law for the Benjamin--Ono equation on the real line~\cite{KM}. This proof does not extend beyond simply connected domains but holds hope to extend to lower regularity. The second proof reduces the problem to the unit disc via the Riemann mapping theorem and to study a boundary operator that up to a conformal transformation appears in a Lax pair for the Benjamin--Ono equation. Here the results of~\cite{KM} applies.  The second proof is conceptually related to the first proof in that the conservation law from~\cite{KM} is, in some sense, an analogue for the Benjamin--Ono equation of Levinson's theorem -- an index theorem describing the number of bound states for Schrödinger operators.

We will throughout the paper use the variational principle, or min-max principle, in the form of Glazman's lemma; see, for instance,~\cite[Theorem 1.25]{FLW}. It says that for a selfadjoint, lower semibounded operator $T$ with corresponding quadratic form $t$ one has
\begin{equation}
\label{eq:glazman}
N(T,0) = \sup\{ \dim \mathcal M:\ t[\psi]< 0 \ \text{for all}\ 0\neq\psi\in\mathcal M \} \,,
\end{equation}
where $\mathcal M$ runs through subspaces of the form domain of $T$.

\section{The disc and radial magnetic fields}\label{discandradial}

In this section we work under the assumption that $\Omega=D(0,R)$, and that the magnetic field $B$ is given by a nonnegative radial function. We abuse the notation a bit and write here \(B(x)=B(|x|)\). It is convenient to switch to polar coordinates, \(x_1=r\cos\theta, x_2=r\sin\theta\). The magnetic vector potential $\Ab$ can be expressed as\footnote{This magnetic potential $\Ab$ is the one that satisfy~\eqref{eq:curl-A=B}.}
\[
    \Ab(r,\theta) = a(r)\,\eb_\theta
    \quad 
    \text{where} 
    \quad
    \eb_\theta
    =
    \begin{pmatrix*}[r]
        -\sin\theta\\ 
        \cos\theta
    \end{pmatrix*}
    \quad
    \text{and}
    \quad
    a(r) = \frac1r \int_0^rB(\rho)\rho \dd \rho .
\]
We will write $a(r) = \phi'(r)$, where $\phi$ is the radial function defined by~\eqref{eq:SuperPot} (in the case of the disc).

We restate the second part of Theorem~\ref{thm:main}, which we are about to prove. We recall that \(\Hb_{\Ab,0}\) denotes the Neumann realization that we are interested in.

\begin{proposition}\label{prop:disc-radial}
Let \(\Omega=D(0,R)\), and assume that the magnetic field \(B = B(r)\) is radial, nonnegative and nonincreasing, with flux \(\Phi\). Then
\[
    N\bigl(\Hb_{\Ab,0},0\bigr) 
    = 
    \AC{\Phi}.
\]
\end{proposition}

\begin{proof}
We perform the usual Fourier series decomposition, and study the family \(\Hm\) of operators in \(L^2((0,R); r\dd r)\), acting as
\[
    \Hm
    =
    -\frac{\dd^2}{\dd r^2} - \frac{1}{r} \frac{\dd}{\dd r} + \left(\frac{m}{r} - \phi'(r)\right)^2 - B(r),
\]
and with Neumann boundary condition at \(r=R\). 
Then $\Hb_{\Ab,0}$ is unitarily equivalent to the orthogonal sum $\oplus_{m \in {\mathbb Z}} \Hm$, and we can count the negative eigenvalues of \(\Hb_{\Ab,0}\) as
\[  
    N\bigl(\Hb_{\Ab,0},0\bigr) =\sum_{m\in\Z} N\bigl(\Hm,0\bigr).
\]
It suffices to show that each \(\Hbm_{\beta}\geq0\) for all \(m<0\), and that for non-negative \(m\) they satisfy
\[
    N(\Hm,0) = 
    \begin{cases}
        1, & 0 \leq m < \Phi, \\
        0, & 0 \leq \Phi \leq m.
    \end{cases}
\]

We first show that \(\Hbm_{\beta}\geq0\) for all \(m<0\). 
For this we use the fact that 
\[
    \phi'(r)= a(r) = \frac1r\int_0^rB(\rho)\rho \dd\rho \geq 0.
\]
Therefore, for \(m<0\),
\[ 
    \left(\frac{m}{r} - \phi'(r)\right)^2 - B(r) \geq \frac{2}{r}\phi'(r) - B(r) = \frac{2}{r^2}\int_0^r \bigl(B(\rho)-B(r)\bigr)\rho\dd \rho.
\]
Since \(B\) is nonincreasing by assumption, this is bounded from below by \(0\).

We next look for zero eigenvalues of \(\Hm\) for \(m\geq 0\). The general solution to the differential equation \(\Hm u=0\) is given by
\[
    u(r) = r^m \ee^{-\phi}\Bigl(c_1 + c_2 \int_r^R \rho^{-1-2m}\ee^{2\phi(\rho)}\dd \rho\Bigr) \qquad (m\geq 0).
\]
To have a solution in the form domain of \(\Hm\), we must take \(c_2=0\), since that term and its derivative are too singular as \(r \to 0^+\). Differentiating, we find (with \(c_1=1\) and \(m\geq 0\))
\[
    u'(R) 
    = 
    \left(\frac{m}{R} - \phi'(R)\right) R^m \ee^{-\phi(R)}.
\]
We recall further that \(\phi(R)=0\), and that \(\phi'(R)=\Phi/R\) by the formula for \(\phi'\) above. Therefore, \(u'(R) = (m-\Phi)R^{m-1}\). It follows that \(u\) satisfies the Neumann boundary condition if, and only if, \(m = \Phi\).

Let us next introduce a variation of the magnetic field. If we replace \(B\) by \(\beta B\) and \(\phi\) by \(\beta\phi\), then we also replace \(\Hm\) by an analytic family of operators \(\Hm(\beta)\), which according to the calculations we just did have a zero eigenvalue precisely when \(\beta = m/\Phi\). This eigenvalue must be the lowest one of each operator, since the corresponding eigenspaces are spanned by \(r^me^{-\beta\phi}\), with constant sign. The calculation also shows that the eigenvalue is necessarily simple. We let  \(u_m\) be a normalized such eigenfunction, and also denote by \(\lambda(\beta)\) the lowest eigenvalue of \(\Hm(\beta)\). The Feynman--Hellmann formula gives
\[
   \lambda'(\beta) = -\int_0^R \left[2\left(\frac{m}{r}-\beta\phi'(r)\right)\phi'(r) + B(r)\right] |u_m|^2 r\dd r.
\]
In particular, for \(\beta = m/\Phi\) the term inside the square brackets equals
\[
    \frac{2m}{\Phi r}\left(\Phi-r\phi'(r)\right)\phi'(r) + B(r)
    =
    \frac{2 m \phi'(r)}{\Phi r}\int_r^R B(\rho)\rho\dd\rho + B(r).
\]
Since $B\geq 0$ and $\phi'\geq 0$, this expression is $\geq 0$. Moreover, it is not identically zero, unless $B$ is identically equal to zero. Since $|u_m|^2$ is positive in $(0,R]$, we deduce that at \(\beta=m/\Phi\), where \(\Hm(\beta)\) has a zero eigenvalue, the derivative \(\lambda'(m/\Phi) < 0 \), unless $B\equiv 0$. This implies that if \(\Phi > m\) then \(\Hm\) has precisely one negative eigenvalue and if \(\Phi\leq m\) it has no negative eigenvalues. Note that the latter conclusion is correct in the trivial case $B\equiv 0$ as well.
\end{proof}

\begin{remark}
The condition that \(B\) should be nonincreasing looks artificial, but it might not be possible to remove. For \(m=-1\), \(R=1\), \(B(r)=(r-1/2)^2+\delta\), with \(\delta>0\), the potential \((\frac{m}{r}-\phi'(r))^2-B(r)\) is no longer positive. For \(r=1\) we get the value \(1-4/(1+12\delta)^2\), which is negative as long as \(\delta<1/12\). The operator might still be bounded from below by zero.
\end{remark}

\begin{remark}
It would be interesting to also count the number of eigenvalues below higher Landau levels. We will then get contributions from both components of the Pauli operator. In the case of a disc with a uniform magnetic field, the calculations can be done explicitly, at least for the second Landau level. We hope to come back to the problem for more general magnetic fields and domains.
\end{remark}

\section{A useful identity for the quadratic form}\label{sec:useful}

We present in this section an identity that links the study of the Schr\"odinger operator \(\Hb_{\Ab,g}\) to a boundary operator. The corresponding identity in the case of \(\Omega = \R^2\) has been used extensively before, for example in the proof of the Aharonov--Casher theorem~\cite{AC}  or in~\cite{ervo}, where it was used to define the Pauli operator for very singular magnetic fields. In the case of bounded domains \(\Omega\), it has been used in the case of Dirichlet boundary conditions~\cite{ekp}, but there the boundary term, which is crucial for us, vanishes. 

For each  $u\in H^1(\Omega)$, we introduce the current
\begin{equation*}
    \jb_\Ab(u)=\re \IP[\big]{ u,(-\ii\nabla - \Ab)u }_{\C}\coloneqq \begin{pmatrix}
        \re\IP[\big]{ u,(-\ii\partial_{x_1} - A_1)u }_{\C}\\[2pt]
        \re\IP[\big]{ u,(-\ii\partial_{x_2} - A_2)u }_{\C}
    \end{pmatrix}.
\end{equation*}
If we write $u=\rho \, \ee^{i\varphi}$ in polar form, which is possible locally away from the zeros of $u$, then we observe that \(\jb_B(u)=\rho^2(\nabla\varphi - \Ab)\).

In the following we denote the norm and inner product in $L^2(\Omega)$ by $\norm{\cdot}$ and $\IP{\cdot,\cdot}$ respectively, with the convention that the inner product is linear in the first entry. We also identify \(x=(x_1,x_2)\in\R^2\) with \(z=x_1+\ii x_2\in\C\), and introduce the Cauchy--Riemann operators (Wirtinger derivatives)
\begin{equation*}
    \partial_{\bar z}
    =
    \frac12\left(\partial_1+\ii\partial_2\right),
    \quad 
    \partial_{z}
    =
    \frac12\left(\partial_1-\ii\partial_2\right).
\end{equation*}
We assume that the magnetic potential \(\Ab\) is the one introduced in~\eqref{eq:def-A}.
\begin{proposition}\label{prop:identity}
The quadratic form \(q_{\Ab,g}\), corresponding to \(\Hb_{\Ab,g}\), can be written as
\[
    q_{\Ab,g}(u)=4\norm{\ee^{-\phi}\partial_{\bar z}\ee^{\phi}u}^2+\int_{\Gamma} \left(g|u|^2+ \tau\cdot\jb_{\Ab}(u)\right)\dd s(x),
    \qquad(u\in H^1(\Omega)),
\]
where the $\jb_{\Ab}(u)$-term in the boundary integral is the pairing of $H^{1/2}(\Gamma)$ and \(H^{-1/2}(\Gamma)\).
\end{proposition}

\begin{proof}
By density it suffices to consider the case $u \in H^2(\Omega)$.
Let  $X_1=-\ii\partial_1 - A_1$ and $X_2=-\ii\partial_2 - A_2$,  where $\Ab=(A_1,A_2)=(-\partial_2\phi,\partial_1\phi)$, and let $v=\ee^{\phi}u$. Then
\begin{equation}\label{eq:identity}
    4\int_\Omega \ee^{-2\phi}\abs{\partial_{\bar z}v}^2 \dd x\\
    =\norm{X_1u}^2 + \norm{X_2u}^2
    -\ii\big(\IP{ X_1u, X_2u }-\IP{ X_2u, X_1u }\big).
\end{equation}
Integration by parts yields
\[
    \begin{aligned}
        \IP{ X_1u,X_2u } 
        &= 
        \IP{ u,X_1X_2u } + \ii\int_\Gamma u\,\overline{(-\ii\partial_2 - A_2)u}\,\nu_1\,\dd s(x),\\
        \IP{ X_2u, X_1u },
        &= 
        \IP{ u, X_2X_1u } + \ii\int_\Gamma u\,\overline{(-\ii\partial_1 - A_1)u}\,\nu_2\,\dd s(x),
    \end{aligned} 
\]
where $\nu=(\nu_1,\nu_2)\coloneqq \mathrm J\tau$ is the unit inward normal of $\Gamma$ and $\mathrm  J = \begin{psmallmatrix*}[r]0&-1\\1&0\end{psmallmatrix*}$.  To finish the proof, we notice that
\[
    X_1X_2-X_2X_1=\ii (\partial_1A_2-\partial_2A_1)=\ii B\,,
\]
and we take the real part in~\eqref{eq:identity}.
\end{proof}

\begin{remark}[Boundary twisted Dirac operator]\label{rem:bd-op}
We denote by \(L=|\Gamma|\) the length of  \(\Gamma\), and by \(s\) the arc-length coordinate along \(\Gamma\), so that
\[
    \tau\cdot(-\ii\nabla - \Ab)|_{\Gamma}=-\ii\partial_s - A_\tau
\]
is a twisted one-dimensional Dirac operator. Here \(A_\tau(s)\coloneqq \tau\cdot\Ab|_{\Gamma}\) is  \(L\)-periodic, and  the boundary term in Proposition~\ref{prop:identity} can be expressed as follows
\[
    g|u|^2 + \tau\cdot\jb_{\Ab}(u)
    =
    \IP{ (-\ii\partial_s - A_\tau+g)u,u}_{\C}.
\]
\end{remark}

\begin{remark}[The  disc and uniform magnetic field]\label{rem:disc}
Assume that $\Omega=D(0,1)$ and $B=1$. Consider the functions
\begin{equation*}
    v_{m}(z)
    =
    z^m
    \quad 
    (z=x_1+\ii x_2\,\text{ and }m\in \N \cup \{0\})\,.
\end{equation*}
Then \(\partial_{\bar z}v_m=0\) and, with \(u_m = \ee^{-\phi}v_m\),
\[
    \int_{\Gamma} \tau\cdot\jb_\Ab(u_m)\dd s(x)
    =
    \int_0^{2\pi}\ee^{\ii m\theta}\overline{\left(-\ii\frac{\dd}{\dd\theta}-\frac{1}{2}\right)\ee^{\ii m\theta}}\dd\theta
    =
    2\pi\left(m-\Phi\right).
\]
\end{remark}

\begin{remark}[The second component of the Pauli operator]\label{rem:identity-c}
Consider the quadratic form
\[
    q_{\Ab,g}^c(u)
    \coloneqq
    \int_\Omega\left(|(-\ii\nabla-\Ab)u|^2 + B|u|^2\right)\dd x
    +
    \int_\Gamma g|u|^2\dd s(x)
\]
corresponding to the Robin realization of the second component of the Pauli operator, \((-\ii\nabla - \Ab)^2 + B\). We have an analogous identity to the one in Proposition~\ref{prop:identity},
\[ 
    q_{\Ab,g}^c(u)
    =
    4\norm{\ee^{\phi}\partial_{z}\ee^{-\phi}u}^2
    +
    \int_{\Gamma} \left(g|u|^2 - \tau\cdot\jb_{\Ab}(u)\right)\dd s(x).
\]
This could be of interest for instance in case of strictly negative magnetic fields.
\end{remark}

\section{Boundary twisted  Dirac operator}\label{sec:Dirac}

Let us assume  that \(\Gamma\) is connected (corresponding to \(\Omega\) being simply connected).  We will discuss later the modifications when \(\Gamma\) has several connected components (see Remark~\ref{rem:Dirac-dc} below). 

Let $V\in C(\Gamma,\R)$. We will later make different choices of \(V\), in terms of \(A_\tau\), \(g\), and the curvature \(\kappa\). We consider the twisted Dirac operator on \(L^2(\Gamma)\) defined via the arc-length coordinate \(s\in[0,L)\) along \(\Gamma\) as follows,
\begin{equation}\label{eq:def-Dirac-V}
    \Dd_V=-\ii\partial_s+V,
\end{equation}
with domain \(H^1(\Gamma)=\{u\in H^1\big((0,L)\big), ~u(0)=u(L)\}\). Here \(L\) denotes the length of \(\Gamma\). The spectrum of \(\mathscr D_V\) is known explicitly.

\begin{proposition}\label{speckkned}
The spectrum of $\Dd_V$ is given by the simple, discrete eigenvalues
\[
    \mu_m
    =
    \mu_m(V,L)
    \coloneqq 
    \frac{2\pi}{L}\left(m+\Phi_V\right) ,\quad m\in\Z,
\]
where
\[
    \Phi_V=\frac1{2\pi}\int_\Gamma V\dd s.
\]
Moreover,  the corresponding eigenfunctions
\[
    f_{m,V}(s)
    \coloneqq 
    L^{-1/2}\exp\left(\ii \mu_m(V,L) s-\ii\int_0^s V(\varsigma)\dd \varsigma \right),
\]
constitute a Hilbert basis of \(L^2(\Gamma)\).
\end{proposition}

\begin{proof}
The spectrum of $\Dd_V$ is characterized by the scalars $\mu$ for which there is a non-zero solution to 
\[
    \begin{cases}
        -\ii f'(s)+V(s)f(s)=\mu f(s),& s\in (0,L),\\
        f(0)=f(L).
    \end{cases}
\]
The general solution to the previous differential equation is 
\[
    f(s)
    =
    C\exp\left(\ii\mu s-\ii\int_0^s V(\varsigma)\dd \varsigma \right).
\]
Therefore, a non-zero solution satisfying \(f(0)=f(L)\) exists if and only if there is an integer $m\in \Z$ such that 
\[
    \mu L-\int_0^L V(\varsigma)\dd \varsigma
    =
    2\pi m.
\]
The previous condition on \(\mu\) reads as follows
\[
    \mu
    =
    \mu_m
    \coloneq
    \frac{2\pi}{L}\left(m+\frac{1}{2\pi}\int_0^L V(s)\dd s\right) = \frac{2\pi}{L}\left(m + \Phi_V\right).
    \qedhere
\]
\end{proof}

\begin{remark}\label{rem:op-D-V}~
\begin{enumerate}
\item It follows from Proposition~\ref{speckkned} that
\begin{equation*}
    \dim \big( \Ker (\Dd_V)\big)
    =
    \begin{cases}
        0, & \text{if }\Phi_V\not\in\Z,\\
        1, & \text{if }\Phi_V\in\Z.
    \end{cases}
\end{equation*}

\item We will encounter \(\Dd_V\) in two typical situations where either
\begin{itemize}
\item \(V= -\tau\cdot\Ab|_{\Gamma}+g  \) and \(\Phi_V=-\Phi+\Phi_g\); or
\item\(V= -\tau\cdot\Ab|_{\Gamma}+g -\kappa \), with \(\kappa\) being the curvature along \(\Gamma\),  and \(\Phi_V=-\Phi+\Phi_g+1\).
\end{itemize}

\item If \(V_1,V_2\in C(\Gamma;\R)\) such that \(\Phi_{V_2}-\Phi_{V_1}\in\Z\),  then \(\Dd_{V_1}\) and \(\Dd_{V_2}\) have the same spectra but the eigenfunctions of \(\Dd_{V_2}\) are modified from those of \(\Dd_{V_1}\) by a pure phase term.
\end{enumerate}
\end{remark}

We define the spectral projection (in \(L^2(\Gamma))\)
\begin{equation}\label{eq:def-proj}
    \Pi_V\coloneqq \Pi_V(0)
    = 
    \mathbf 1_{(-\infty, 0)}(\Dd_V),
    \qquad
    \Pi_V^{^\leq} \coloneqq \Pi_V^{^\leq} (0)
    =
    \mathbf 1_{(-\infty, 0]}(\Dd_V), 
\end{equation}
on the eigenspace of \(\Dd_V\) corresponding to the negative (resp. non-positive) eigenvalues.

\begin{lemma}
\label{charpiv}
Whenever $V\in C(\Gamma,\R)$ we define
\[
    \Theta_V(s)
    \coloneqq
    \frac{2\pi}{L}\Phi_Vs-\int_0^sV(\varsigma)\dd\varsigma.
\]
For any $\alpha\in \R$, the spectral projections
\[
    \Pi_V(\alpha)
    \coloneqq
    \mathbf 1_{(-\infty, \alpha)}(\Dd_V),
\]
are of the form 
\begin{equation}\label{ideinfintofpiv}
    \Pi_V(\alpha)
    =
    \ee^{\ii\Theta_V}\Pi_0\left(\alpha-\frac{2\pi}{L}\Phi_V\right)\ee^{-\ii\Theta_V}.
\end{equation}
Moreover, if $V\in C^\infty(\Gamma,\R)$ then $\Pi_V(\alpha)$ is a classical pseudodifferential operator of order zero and the principal symbol of $\Id - \Pi_V(\alpha)$ coincides with that of the Calder\'on projector of $\ee^{-\phi}\partial_{\bar{z}}\ee^\phi$.
\end{lemma}

\begin{proof}
By Proposition~\ref{speckkned},  the eigenfunctions of \(\Dd_V\) can be written in the form
\[
    f_{m,V}(s)=\ee^{\ii \Theta_V(s) }f_{m,0}(s).
\]
We conclude~\eqref{ideinfintofpiv} from the equalities
\[
\ee^{-i\Theta_V}\Pi_V\left(\alpha\right)\ee^{i\Theta_V}=
    \sum_{\substack{m<L\alpha/2\pi -\Phi_V\\m\in\Z}}\langle \cdot,f_{m,0}\rangle f_{m,0}=\Pi_0\left(\alpha-\frac{2\pi}{L}\Phi_V\right).
\]
We will now show that $1-\Pi_V(\alpha)$ is a classical pseudodifferential operator of order zero when $V\in C^\infty(\Gamma,\R)$. By~\eqref{ideinfintofpiv}, it suffices to consider the case $V=0$. Conjugating by the unitary $u(s):=\ee^{2 \pi i N s/L}$ we can reduce to the case $\alpha\in (-\frac{2\pi}{L},0]$. In this case, a computation with the geometric series shows that for $f\in C^\infty(\Gamma)$
\begin{equation}\label{compoar}
    (1-\Pi_0(\alpha))f(\ee^{is})
    =
    \lim_{r\to 1^-}\frac{1}{L}\int_0^L \frac{f(\ee^{i\zeta})}{1-r\ee^{i(s-\zeta)}}\dd\zeta.
\end{equation}
The arc length parametrization defines a diffeomorphism $\Gamma\cong \frac{L}{2\pi }S^1$ allowing us to identify $C^\infty(\Gamma)$ with $C^\infty(\frac{L}{2\pi }S^1)$. Under this identification,~\eqref{compoar} shows that $1-\Pi_0(\alpha)$ is the classical pseudodifferential operator of order zero on the circle defined from the Cauchy integral that up to lower order terms coincides with the Calder\'on projector.
\end{proof}

Recall that \((\tau,\nu)\) is an orthonormal direct frame along \(\Gamma\) such that \(\nu\) is the inward pointing normal.  Recall that the curvature along \(\Gamma\) is the scalar function \(\kappa\) defined by the relation
\begin{equation*}
    \dot\tau = \kappa\nu,
\end{equation*}
where the dot denotes differentiation with respect to the arc-length coordinate.  

We introduce the Cartesian components of \(\tau,\nu\) as follows
\[
    \tau=(\tau_1,\tau_2),\quad \nu=(\nu_1,\nu_2)
\]
and recall that the relation \(\nu=\begin{psmallmatrix*}[r]0&-1\\1&0\end{psmallmatrix*}\tau\) reads as follows
\begin{equation}\label{eq:nu=J-tau}
    \nu_1 = -\tau_2,
    \quad 
    \nu_2 = \tau_1.
\end{equation}
There is a nice duality between projections of the type \(\Pi_V\) and \(\Pi_V^{^\leq} \) that we explain below.
\begin{proposition}\label{prop:duality-proj}
Let  \(w\in L^2(\Gamma)\).  Then it holds that
\[
    w\in \big(\im(\Pi_{V-\kappa})\big)^\perp
\]
if, and only if, the function \(h\coloneqq (\nu_1-\ii\nu_2)w\) satisfies
\[
    h\in \big(\im(\Pi_V)\big)^\perp.
\] 
\end{proposition}

Here and below we denote by $\im(T)$ the range of a linear operator $T$. 

\begin{proof}
We identify the vectors \(\tau\) and \(\nu\) with the complex number \(\tau_1+\ii\tau_2\) and \(\nu_1+\ii \nu_2\), respectively. The identities in~\eqref{eq:nu=J-tau} then take the form
\[ 
    \dot\tau=\ii\kappa \tau,
\]
which yields that 
\[
    \tau(s)
    =
    \tau(0)\exp\left(\ii\int_0^s \kappa(\varsigma)\dd \varsigma\right),
    \quad 
    \nu(s) = i \tau(s)
    =
    \ii\tau(0)\exp\left(\ii\int_0^s \kappa(\varsigma)\dd \varsigma\right).
\]
Assume that \(h\in \big(\im(\Pi_V)\big)^\perp\). Then, by Proposition~\ref{speckkned}, 
\[
    h\perp f_{m,V} \quad (m\in\Z,\,\mu_m(V,L)<0).
\]
In other words, $\langle w, \tilde f_{m,V}\rangle=0$, where $\tilde f_{m,V}(s)=\nu f_{m,V}(s)$.
Since \(\int_\Gamma\kappa\dd s=2\pi\) by the Gauss--Bonnet theorem,  it is easy to check that
\[
    \mu_m(V,L)=\mu_{m+1}(V-\kappa,L)
\]
and
\[
    \begin{aligned}
        \tilde f_{m,V}(s) 
        & = \ii \tau(0) L^{-1/2}\exp\left(\ii\mu_{m+1}(V-\kappa,L) s
        - \ii\int_0^s \big(V(\varsigma)
        - k(\varsigma)\big)\dd \varsigma\right)\\
        & = \ii \tau(0) f_{m+1,V-\kappa}(s).
    \end{aligned}
\]
This proves that 
\[
    w\perp f_{m+1,V-\kappa}\qquad(m\in\Z,\mu_{m+1}(V-\kappa,L)<0),
\]
which yields that \(w\in \big(\im(\Pi_{V-\kappa})\big)^\perp\).
In a similar fashion, we can prove that \(h\in \big(\im(\Pi_V)\big)^\perp\) if we know that \(w\in \big(\im(\Pi_{V-\kappa})\big)^\perp\).
\end{proof}

\begin{remark}\label{rem:Dirac-dc}
In the  case where \(\Gamma\) consists of \(d+1\) curves \((\Gamma_j)_{j=0}^d\), with \(d\geq1\) (which corresponds to \(\Omega\) non-simply connected),  we can view the Hilbert space \(L^2(\Gamma)\) as the direct sum \(\dsum L^2(\Gamma_j)\) and introduce the operators \(\mathcal D_{V_j}\) on \(L^2(\Gamma_j)\cong L^2([0,L_j))\),  where we  denote by \(L_j\) the length of \(\Gamma_j\).
We introduce the operator \(\mathcal D_V\) on \(L^2(\Gamma)\) as the direct sum \(\dsum  \mathcal D_{V_j}\).  We then have
\begin{itemize}
\item The spectrum of \(\mathcal D_V\) is \(\{\mu_m(V_j,L_j)\}_{m\in\Z,0\leq j\leq d}\),  the union of the spectra of \(\mathcal D_{V_j}\) described in Proposition~\ref{speckkned},  and an orthonormal basis of eigenfunctions is given by \(\{f_{m,V_j}\}_{m\in\Z,0\leq j\leq d}\),  where \(f_{m,V_j}\),  defined initially on \(\Gamma_j\),  is viewed in \(L^2(\Gamma)\) by extension by \(0\).

\item The spectral projection \(\Pi_V=\mathbf 1_{(-\infty,0)}(\Dd_V)\) is also given as the direct sum
\[
    \Pi_V=\dsum \Pi_{V_j}
\]

\item Proposition~\ref{prop:duality-proj} continues to hold.
\end{itemize}
\end{remark}

\section{Eigenmodes with holomorphic/anti-holomorphic conditions}\label{sec:eginnad}

\subsection{New auxiliary operators}
\label{subsec:newaux}

Let us introduce the trace operator
\begin{equation*}
    \gamma_0: u\in H^1(\Omega) \mapsto u|_{\Gamma}\in H^{1/2}(\Gamma)
\end{equation*}
and the tangential component of the magnetic potential in~\eqref{eq:def-A}, 
\begin{equation}\label{eq:def-a}
    A_\tau = \tau\cdot \Ab|_{\Gamma}.
\end{equation}
The trace operator $\gamma_0$ has been implicit in the presentation so far, but for clarity it will be important to include it in formulas going forward. In light of Proposition~\ref{prop:identity} and Remark~\ref{rem:bd-op}, we have the following identity,
\begin{equation}\label{eq:def-q-1}
    q_{\Ab,g}(u)
    =
    4\|\ee^{-\phi}\partial_{\bar{z}}\ee^{\phi}u\|^2_{L^2(\Omega)}+\int_{\Gamma} (-\ii\partial_s-A_\tau+g)(\gamma_0u)\overline{\gamma_0u} \dd s,
\end{equation}
for \(u\in H^1(\Omega)\).  The boundary term in~\eqref{eq:def-q-1} is  the pairing of $(-\ii\partial_s-A_\tau +g)(\gamma_0 u)\in H^{-1/2}(\Gamma)$ with $\gamma_0 u\in H^{1/2}(\Gamma)$.  Let us denote by \(\mathcal O(\Omega)\) the space of holomorphic functions in \(\Omega\). The Bergman space
\[
    \mathscr B(\Omega)
    \coloneqq
    \mathcal{O}(\Omega) \cap L^2(\Omega)
\]
is a closed subspace of \(L^2(\Omega)\), and the quadratic form
\begin{equation*}
    \mathscr B(\Omega) \cap  H^1(\Omega)
    \ni v
    \mapsto 
    q_{\Ab,g}^{\mathcal O}(v)
    \coloneqq 
    \int_{\Gamma} (-\ii\partial_s-A_\tau +g)(\gamma_0 v)\overline{\gamma_0 v} \dd s
\end{equation*}
is semi-bounded and closed\footnote{It inherits these properties from \(q_{\Ab,g}\), see~\eqref{eq:identity}.}, so it induces a self-adjoint operator, \(\mathscr H_{\Ab,g}^{\mathcal O}\), in the  Hilbert space \(\mathscr B(\Omega)\). 

Let us observe that, by~\eqref{eq:def-q-1} and the property \(\phi=0\) on \(\Gamma\),  we have the identity,
\[
    q_{\Ab,g}(u)=\int_{\Gamma} (-\ii\partial_s-A_\tau +g)(\gamma_0 u)\overline{\gamma_0 u} \dd s
    =
    q^{\mathcal O}_{\Ab,g}(\ee^{\phi} u), 
    \quad u\in H^1(\Omega)\cap\ee^{-\phi}\mathcal{O}(\Omega).
\]
Thus, by Glazman's lemma~\eqref{eq:glazman},   
\begin{equation}\label{eq:N->N-O}
    N(\Hb_{\Ab,g},0)
    \geq 
    N(\mathscr H^{\mathcal O}_{\Ab,g},0),
\end{equation}
and the question of finding a lower bound on the number of negative eigenvalues of \(\Hb_{\Ab,g}\),  becomes to bound \(N(\mathscr H^{\mathcal O}_{\Ab,g},0)\) from below.  We shall see that the latter quantity will be bounded from below by the index of a certain Atiyah--Patodi--Singer realization of the Pauli--Dirac operator (see Remark~\ref{rem:index}).

There is another operator in $L^2(\Gamma)$ obtained from imposing the anti-holomorphic condition. Consider the closed and semi-bounded quadratic form defined on $H^1(\Omega)$ as follows
\begin{equation}\label{conjnjnad}
    {q}_{\Ab,g}^c(u)
    =
    4\|\ee^{\phi}\partial_{z}\ee^{-\phi}u\|^2_{L^2(\Omega)}
    +
    \int_{\Gamma} (\ii\partial_s+A_\tau +g)(\gamma_0u)\overline{\gamma_0u} \dd s,
\end{equation}
(note the $\partial_{z}$ in~\eqref{conjnjnad} instead of $\partial_{\bar z}$ in~\eqref{eq:def-q-1})
and denote by \(\Hbc_{\Ab,g}\) its corresponding self-adjoint operator.  Note that, by Remark~\ref{rem:identity-c} this corresponds to the second component of the Pauli operator in~\eqref{eq:Pauli}; in fact
\[
    \Hbc_{\Ab,g} = (-i\nabla-\Ab)^2+B
\]
with the Robin boundary condition in~\eqref{eq:bc-Robin}.  

We can encounter negative eigenvalues of \((-i\nabla-\Ab)^2+B\) corresponding to eigenmodes, $u$, of the quadratic form such that $\mathrm{e}^{-\phi}u$ is antiholomorphic in $\Omega$.  More precisely, we denote by \(\bar{\mathcal O}(\Omega)\)  the space of anti-holomorphic functions in \(\Omega\) and consider the closed, semibounded quadratic form defined as
\begin{equation}\label{eq:def-qf-H-op-c}
    \bar{\mathscr B}(\Omega)
    \cap 
    H^1(\Omega)
    \ni 
    v
    \mapsto
    \mathfrak \mathfrak{q} ^{\bar{\mathcal O}}_{\Ab,g}(v)\coloneqq \int_{\Gamma} (\ii\partial_s+A_\tau +g)(\gamma_0 v)\overline{\gamma_0 v} \dd s,
\end{equation}
where $\bar{\mathscr{B}}(\Omega)=\bar{\mathcal O}(\Omega)\cap L^2(\Omega)$ is the space of square integrable, anti-holomorphic functions. 
Denote by \(\Hcc\) the self-adjoint operator in the Hilbert space \(\bar{\mathscr{B}}(\Omega)\),  corresponding to the foregoing quadratic form.   
The identity in Remark~\ref{rem:identity-c} and Glazman's lemma~\eqref{eq:glazman} yield, as above,
\[
    N(\Hbc_{\Ab,g},0)
    \geq
    N(\Hcc,0).
\]

\subsection{Example: The case of the disc}
Let us illustrate how an explicit computation of \(N(\mathscr H^{\mathcal O}_{\Ab,g},0)\) is more feasible than the number of negative eigenvalues of \(\Hb_{\Ab,g}\). We discuss here the case of the unit disc in further detail; its symmetric nature makes it more tractable for computations.

For the sake of simplicity, we impose a Neumann boundary condition, i.e., we set \(g=0\). Combining Proposition~\ref{prop:disc-radial} and Remark~\ref{rem:disc},  we get the following proposition.

\begin{proposition}\label{prop:disc-cst}
If the magnetic field, \(B\),  is a non-negative constant and \(\Omega\) is the open  unit disc, \(D(0,1)\), then 
\[
    N (\Hb_{\Ab,0},0)
    =
    N(\mathscr H^{\mathcal O}_{\Ab,0},0)
    =
    \AC{\Phi},
    \quad
    N(\mathscr H^{\bar{\mathcal O}}_{\Ab,0},0_+)
    = 
    \begin{cases} 
        1 & \text{if } B=0,\\
        0 & \text{if } B>0,
    \end{cases}
\]
where \(\AC{\cdot}\) is introduced in~\eqref{eq:def-K} and \(\Hb_{\Ab,0}\) is the operator with a Neumann boundary condition.
\end{proposition}

\begin{proof}
If \(g=0\) and \(B>0\), then the quadratic form in~\eqref{eq:def-qf-H-op-c} is positive by Remark~\ref{rem:identity-c},  hence \(N(\mathscr H^{\bar{\mathcal O}}_{\Ab,0},0_+)=0\); if \(B=0\) then the constant functions are the zero-modes of  the quadratic form in~\eqref{eq:def-qf-H-op-c},  hence \(N(\mathscr H^{\bar{\mathcal O}}_{\Ab,0},0_+)=1\).

For \(0\leq m< \AC{\Phi} \), the functions $e^{-\phi} z^m$ produce negative eigenvalues of \(\mathscr H^{\mathcal O}_{\Ab,0}\). Therefore, $N(\mathscr H^{\mathcal O}_{\Ab,0},0)\geq \AC{\Phi}$. Combining Proposition~\ref{prop:disc-radial}  with Equation~\eqref{eq:N->N-O} we have that 
\[
    \AC{\Phi}
    \leq 
    N(\mathscr H^{\mathcal O}_{\Ab,0},0)
    \leq N (\Hb_{\Ab,0},0)
    =
    \AC{\Phi},
\]
and, consequently, \(N (\Hb_{\Ab,0},0)=N(\mathscr H^{\mathcal O}_{\Ab,0},0)\).
\end{proof} 

The proof of Proposition~\ref{prop:disc-cst} indicates an easy way to obtain a lower bound for \(N (\Hb_{\Ab,0},0)\). Even in the case of the unit disc,  \(\Omega=D(0,1)\),  we do not have such an \emph{easy} proof for the lower bound stated in Theorem~\ref{thm:main}, unless the tangential component  of the magnetic potential \(\Ab\) is constant on the boundary of $\Omega$. 

\subsection{General domains}
\label{subsec:gendom}

There is an interesting relation between the number \(N(\Hcr,0)\) of negative eigenvalues
of \(\Hcr\) and that of a certain restriction of the boundary Dirac operator encountered in Section~\ref{sec:Dirac}.  A similar relation  exists for  the number \(N(\Hcc,0_+)\) of non-positive eigenvalues of \(\Hcc\). 

We introduce the boundary potentials
\begin{equation}\label{eq:def-V}
    V = g-A_\tau,
    \quad 
    V^c = -g-A_\tau,
\end{equation}
where \(A_\tau\) is introduced in~\eqref{eq:def-a} and \(g\) is the function defining the Robin condition in~\eqref{eq:bc-Robin}.

We introduce the eigenspaces
\begin{equation}\label{eq:def-E-V}
    E_V = \im (\Pi_V) = \im( \mathbf 1_{(-\infty, 0)}(\Dd_V)),
    \quad 
    F_{V^c} 
    = \im\big(\mathrm{Id}-\Pi_{V^c}\big)
    = (E_{V^c})^{\bot},
\end{equation}
where \(\Pi_{\circ}\) is the spectral projection defined in~\eqref{eq:def-proj}.  Note that \(E_V\) corresponds to the negative eigenvalues of the operator \(\Dd_V\) introduced in~\eqref{eq:def-Dirac-V},  while  \(F_{V^c}\) corresponds to the non-negative eigenvalues of \(\Dd_{V^c}\).  Note that these objects are defined as direct sums in the non-simply connected case (see Remark~\ref{rem:Dirac-dc}). In particular we get,
\[
    E_V=\dsum  \im( \mathbf 1_{(-\infty, 0)}(\Dd_{V_j}))=\dsum E_{V_j},
\]
where the last identity defines the subspace \(E_{V_j}\) of \(L^2(\Gamma_j)\).

\begin{proposition}\label{knlknad*}
It holds that
\[
    N(\Hcr,0)
    \geq 
    \dim\bigl(E_V\cap \gamma_0(\mathscr B(\Omega))\bigr)
\]
and
\[
    N(\Hcc,0_+)
    \geq 
    \dim \bigl(F_{V^c}\cap \gamma_0(\bar{\mathscr B}(\Omega))\bigr).
\]
\end{proposition}

\begin{proof}
Notice first that the trace operator $\gamma_0:\mathscr{B}(\Omega)\to H^{-1/2}(\Gamma)$ is injective (a holomorphic function with zero trace on \(\Gamma\) is zero in all of \(\Omega\)).

By~\cite[(3.2a) in proof of Proposition 3.1]{ASSpair}, $E_V\cap \gamma_0(\mathscr B(\Omega))=\ker(P_{\mathcal{C}}+\Pi_V-2)$ where $P_{\mathcal{C}}$ denotes the Calder\'on projector of $\bar{\partial}$. Lemma~\ref{charpiv} and elliptic regularity implies that the space $E_V\cap \gamma_0(\mathscr B(\Omega))$ is a finite-dimensional subspace of $H^{1/2}(\Gamma)$. Therefore,  if \(f\in E_V\cap \gamma_0(\mathscr B(\Omega))\),  then \(f\in H^{1/2}(\Gamma)\)  and there exists \(F\in \mathscr B(\Omega)\) such that \(\gamma_0(F)=f\).  Since \(F\) is holomorphic in \(\Omega\), \(F\in H^{1}(\Omega)\) by elliptic regularity for $\partial_{\bar{z}}$.  Therefore \(F\) is in the form domain of the operator \(\Hcr\).  
The proof of the inequality 
\( N(\Hcr,0) \geq \dim\bigl(E_V\cap \gamma_0(\mathscr B(\Omega))\bigr)\) now follows by Glazman's lemma.
    
In a similar fashion,  the inequality 
\(N(\Hcc,0_+) \geq \dim \bigl(F_{V^c}\cap \gamma_0(\bar{\mathscr B}(\Omega))\bigr)\) follows.
\end{proof}

\begin{remark}
It is unclear to us when the opposite inequalities in Proposition~\ref{knlknad*} can be expected to hold true. For instance, a natural question is whether we have equalities in Proposition~\ref{knlknad*} when $B\geq 0$. Let us point at the technical difficulty. Let \(N=N(\Hcr,0)\) and  consider an orthonormal basis \(\{u_1,\ldots,u_N\}\) of the eigenspace \(M=\im\big(\mathbf 1_{(-\infty,0)}(\Hcr)\big)\), corresponding to the negative eigenvalues of \(\Hcr\).  Then \(\dim\big(\gamma_0(M)\big)=N\) but we do not necessarily have the inclusion \(\gamma_0(M)\subset E_V\cap \gamma_0(\mathscr B(\Omega))\). Even after projecting again on \(E_V\), we do not know that the set \(\{\Pi_V\gamma_0(u_1),\ldots ,\Pi_V\gamma_0(u_n)\}\) is linearly independent, or even a subset of \(\gamma_0(\mathscr{B}(\Omega))\).  

When $B$ is allowed to change sign, the inequalities in Proposition~\ref{knlknad*} can be strict. Indeed, with \(\Omega = D(0,1)\) and \(B(x) = -4x_2\) the first inequality in Proposition~\ref{knlknad*} is strict. In this setup, $\phi(x)=\frac{x_2}{2}(1-x_1^2-x_2^2)$, \(\Ab(x) = (-\frac{1}{2}(1-x_1^2-3x_2^2),-x_1x_2)  \), \(A_\tau(s)=-\sin(s)\) and \(g=0\) (so \(V(s)=-A_\tau(s)\)),   and \(B\) has zero flux. Let \(v(z) = e^{-\ii z/2}\). Then \(v\) is holomorphic in the interior of \(\Omega\) and
\[
    \begin{aligned}
        q_{B,0}^{\mathcal O}(\ee^{\phi}v) 
        &= 
        \int_0^{2\pi}\bigl[\bigl(-\ii \partial_s+\sin(s)\bigr)v(\ee^{is})\bigr]\overline{v(\ee^{is})}\dd s\\
        &=
        -\frac{1}{2} \int_0^{2\pi} \ee^{-\ii s} \ee^{\cos s}\dd s\\
        &=-\pi I_1(1) <0,
    \end{aligned}
\]
where \(I_1\) denotes the modified Bessel function of the first kind.  By Glazman's lemma~\eqref{eq:glazman},  this yields that \(N(\Hcr,0)\geq 1\).  However, no non-zero linear combintion of the Fourier series of \(f_{m,V}(s)=\ee^{\ii z}\ee^{-\ii\bar{z}}\ee^{\ii(m-1)s}\) is of the form 
\(\sum_{n\geq 0}\ee^{\ii ns}\). Therefore, no non-zero linear combination of the eigenfunctions of \(\Dd_{V}\), 
\[
    f_{m,V}(s)=\exp(\ii(m-1)s)\exp\left(\ii\frac{\ee^{\ii s}+\ee^{-\ii s} }{2}\right)\not\in \gamma_0(\mathscr B(\Omega))
\]
is a boundary restriction of a holomorphic function, so \(\dim\bigl(E_{V}\cap \gamma_0(\mathscr B(\Omega))\bigr)=0\).
\end{remark}

\section{Proof of Theorem~\ref{thm:main} using index theory}
\label{sec:index}

We present the proof of Theorem~\ref{thm:main} by assuming first  that the magnetic potential  \(\Ab\) is given by~\eqref{eq:def-A}.  The proof in the general case of any smooth magnetic potential will require a few  adjustments that we deal with in Remark~\ref{rem:generalmagn} below.

\subsection{An intermezzo on Dirac operators}
\label{subsec:intermezzo}

To describe the negative eigenvalues in terms of index theory, we shall need a suitable formalism. We consider the Dirac operator
\begin{equation*}
    \slashed{D}_\phi 
    = 
    \begin{pmatrix} 
        0                  & \slashed{D}_\phi^-\\
        \slashed{D}_\phi^+ & 0
    \end{pmatrix}
    \coloneqq 
    2\begin{pmatrix} 
    0                                     & -\ee^{\phi}\partial_{z}\ee^{-\phi} \\
    \ee^{-\phi}\partial_{\bar z}\ee^{\phi}& 0
    \end{pmatrix}
\end{equation*}
acting on functions in $\Omega$. The operator $\slashed{D}_\phi$ is formally self-adjoint and forms a \enquote{square root} of the full Pauli operator,
\begin{equation*}
    \slashed{D}_\phi^2
    =
    \begin{pmatrix} 
        (-\ii\nabla -\Ab)^2- B & 0                       \\
        0                      &(-\ii\nabla - \Ab)^2 + B
    \end{pmatrix}.
\end{equation*}

We next introduce a realization \(\slashed{D}_{\phi,\mathrm{APS},V}^{+}\) of \(\slashed{D}_\phi^{+}\), in \(L^2(\Omega)\), whose index can be calculated with the Atiyah--Patodi--Singer index theory~\cite{APS,boosss,ggrubb,melrose}. The domain of the operator $\slashed{D}_{\phi,\mathrm{APS},V}^+$ is
\begin{equation}\label{eq:def-dom-APS}
    \mathrm{Dom}\big(\slashed{D}_{\phi,\mathrm{APS},V}^{+}\big)
    =
    \mathfrak D_V(\Omega)\coloneqq \{u\in H^1(\Omega): \gamma_0(u)\in E_V\},
\end{equation}
where \(V\) is the potential introduced in~\eqref{eq:def-V} and \(E_V\) is the eigenspace from~\eqref{eq:def-E-V}.

\begin{proposition}\label{lnnada}
The operator \(\slashed{D}_{\phi,\mathrm{APS},V}^{+}\) is a Fredholm operator.  Its adjoint $(\slashed{D}_{\phi,\mathrm{APS},V}^+)^*$ is the differential operator $\slashed{D}_\phi^-$ on $\Omega$ with the domain 
\[
    \mathrm{Dom}\bigl((\slashed{D}_{\phi,\mathrm{APS},V}^+)^*\bigr)
    \coloneq
    \{u\in H^1(\Omega): \gamma_0(u) \in (E_{V-\kappa})^{\perp}\}.
\]
Furthermore, the boundary restriction $\gamma_0$ implements isomorphisms
\begin{align}
\label{isorfrom}
    \ker \bigl(\slashed{D}_{\phi,\mathrm{APS},V}^{+}\bigr) 
    &\cong
     E_V \cap \gamma_0(\mathscr{B}(\Omega)), \quad\text{and}\\
     \nonumber
\ker \bigl((\slashed{D}_{\phi,\mathrm{APS},V}^+)^*\bigr)
    &\cong
    (E_{V-\kappa})^{\perp} \cap \gamma_0(\bar{\mathscr{B}}(\Omega)).
\end{align}
\end{proposition}

\begin{remark}\label{rem:index}
    By Propositions~\ref{knlknad*} and~\ref{lnnada}, the index of $\slashed{D}_{\phi,\mathrm{APS},V}^+$ satisfies
    \begin{equation*}
    \ind (\slashed{D}_{\phi,\mathrm{APS}, V}^+) \leq N(\Hcr,0).
    \end{equation*}
\end{remark}    

\begin{proof}[Proof of Proposition~\ref{lnnada}]
By Lemma~\ref{charpiv}, the projection \(\Pi^{\perp}_V\) only differs from the Calder\'on projector for $\slashed{D}_\phi^+$---the Cauchy integral operator---by an operator compact on $H^{1/2}(\Gamma)$ and $H^{-1/2}(\Gamma)$.  It follows from~\cite[Theorem 4]{goffbandara} (see also~\cite[Theorem 7.20]{baerball} and~\cite{baerbandara}) that $\slashed{D}_{\phi,\mathrm{APS},V}^+$ is a Fredholm operator, and even regular.
    
We now compute the adjoint of $\slashed{D}_{\phi,\mathrm{APS},V}^{+}$. Take \(u\) in the domain of \((\slashed{D}_{\phi,\mathrm{APS},V}^{+})^*\). For all \(v\) in the domain of \(\slashed{D}_{\phi,\mathrm{APS},V}^{+}\) we have
\[
    \IP{u,\slashed{D}_{\phi,\mathrm{APS},V}^{+}v}
    =
    \IP{(\slashed{D}_{\phi,\mathrm{APS},V}^{+})^*u,v}.
\]
Integration by parts yields
\[
    \IP{u,\slashed{D}_{\phi,\mathrm{APS},V}^{+}v}
    =
    \IP{\slashed{D}_{\phi}^-u,v}-\IP{(\nu_1-\ii\nu_2)u,v}_{L^2(\Gamma)},
\]
hence  \(u\) is in the domain of \((\slashed{D}_{\phi,\mathrm{APS},V}^{+})^*\)  if, and only if, \((\nu_1-\ii\nu_2) \gamma_0(u) \perp \gamma_0(v)\). By Proposition~\ref{prop:duality-proj} we conclude that \(\gamma_0(u)\in (E_{V - \kappa})^{\perp}\). 
    
A direct verification shows that $\gamma_0$ induces the isomorphisms stated in~\eqref{isorfrom}.
\end{proof}

\subsection{Reduction to standard APS-index}
\label{subsec:relhomo}

We now compute the index of $\slashed{D}^+_{\phi,\mathrm{APS},V}$ introduced in Theorem~\ref{lnnada} by using a relative index theorem that realizes the index as a flow of eigenvalues. It suffices to consider the operator \(\slashed{D}_{\mathrm{APS},V}^{+}\), with action
\[
    \slashed{D}_{\mathrm{APS},V}^{+}
    \coloneqq 
    \slashed{D}_{\phi}^{+}|_{\phi=0} 
    =
    2\partial_{\bar z},
\]
and with domain \(\mathfrak D_V(\Omega)\) from~\eqref{eq:def-dom-APS}. In fact, the identity~\eqref{isorfrom} in Proposition~\ref{lnnada} implies that $\ind (\slashed{D}_{\phi,\mathrm{APS}, V}^+)$ is independent of $\phi$ (since the right hand side of~\eqref{isorfrom} is independent of \(\phi\)) and we have the identity
\begin{equation}\label{eq:index-phi=0}
    \ind (\slashed{D}_{\phi,\mathrm{APS}, V}^+)
    =
    \ind \big(\slashed{D}_{\mathrm{APS}, V}^+\big).
\end{equation}

Recall that the boundary \(\Gamma\) consists of \(d+1\) curves and can be disconnected (when \(d\geq 1\)).  For all \(\vec\lambda\in\R^{d+1}\),  we introduce the  operator 
\begin{equation}\label{eq:def-op-APS*}
    \left.
    \begin{aligned}
        \slashed{D}_{\mathrm{APS}, V}^+(\vec\lambda)
        & \coloneqq 
        2\partial_{\bar z},\\
        \dom \bigl(\slashed{D}_{\mathrm{APS}, V}^+(\vec\lambda)\bigr)&\coloneqq \{u\in H^1(\Omega)~:~\gamma_0(u)\in E_V(\vec\lambda)\},
    \end{aligned}
    \right\}
\end{equation}
where \(E_V(\vec\lambda)\coloneqq \im\bigl(\Pi_V(\vec\lambda) \bigr)\) is the subspace defined by the spectral projection
\begin{equation*}
    \Pi_V(\vec\lambda)
    \coloneqq \dsum
    \mathbf 1_{(-\infty,\lambda_j)}(\Dd_{V_j}),
\end{equation*}
and notice that,  in the simply connected case,  \(d=0\),  the previous formula reduces to
\[    
    \Pi_V(\lambda)
    \coloneqq
    \mathbf 1_{(-\infty,\lambda)}(\Dd_{V}).
\]
In terms of the orthogonal projection on \((E_V(\vec\lambda))^\perp\), \(\Pi_V^\perp(\vec\lambda)\coloneqq \Id - \Pi_V(\vec\lambda)=\dsum \mathbf 1_{[\lambda_j,+\infty)}(\Dd_{V_j})\), we can express the domain of \(\slashed{D}_{\mathrm{APS}, V}^+(\vec\lambda)\) as follows
\[
    \dom \bigl(\slashed{D}_{\mathrm{APS}, V}^+(\vec\lambda)\bigr)
    =
    \{u\in H^1(\Omega)~:~\Pi_V^\perp(\vec\lambda)\gamma_0(u)=0\}.
\]
Using that the index is invariant by homotopy, we will now prove that we can reduce from computing the index with \(V=g-A_\tau\) to \(V=0\), but with the flux appearing in  the spectral parameter \(\tilde{\Phi}_V=(\frac{2\pi}{L_0}\Phi_{V_0},\cdots,\frac{2\pi}{L_d}\Phi_{V_d})\in\R^{d+1}\)  (it is a vector in the non-simply connected case,  see Remark~\ref{rem:Dirac-dc}). If \(\Omega\) is simply connected, \(\tilde{\Phi}_V\) has only one component which is \(2\pi/L\) times the scalar  \(\Phi_V\)
 introduced in~\eqref{eq:def-Phi-g}.

\begin{proposition}\label{lemkmknknk}
It holds that 
\[
    \ind \bigl(\slashed{D}_{\mathrm{APS}, V}^+\bigr)
    =
    \ind \bigl(\slashed{D}_{\mathrm{APS}, 0}^+(-\tilde{\Phi}_V)\bigr).
\]
\end{proposition}

\begin{proof}
Recall that \(L^2(\Gamma)\cong \dsum L^2(\Gamma_j)\) and that \(\Dd_V=\dsum \Dd_{V_j}\).  By Proposition~\ref{speckkned},  the eigenfunction \(f_{m,V_j}\) of \(\Dd_{V_j}\) can be written on the component \(\Gamma_j\cong [0,L_j)\) in the form
\[
    f_{m,V_j}(s)=\ee^{\ii \Theta_{V_j}(s) }f_{m,0}(s),
    \quad\text{where}\quad
    \Theta_{V_j}(s)=\frac{2\pi}{L_j}\Phi_{V_j}s-\int_0^sV_j(\varsigma)\dd\varsigma \quad (0\leq s\leq L_j).
\]
For all \(t\in[0,1]\),  we introduce the orthogonal projection
\[
    P_t \coloneqq\dsum P_{t,j},\quad P_{t,j} \coloneqq \sum_{\substack{m\geq -\Phi_{V_j}\\m\in\Z}} \IP{\cdot,f_{m,tV_j}}_{L^2(\Gamma_j)} f_{m,tV_j}\,,
\]
on the closure of the space spanned by 
\[
    S_{tV}\coloneqq\dsum S_{tV_j},\quad S_{tV_j}\coloneqq \{f_{m,tV_j}~:~ m\in\Z,m\geq-\Phi_{V_j} \}.
\]
Notice that
\[
    \begin{gathered}
        P_{0,j}
        =
        \Pi_0^\perp\Bigl(-\frac{2\pi}{L_j}\Phi_{V_j}\Bigr),\quad 
        P_{1,j}
        =
        \Pi_{V_j}^\perp(0),\\
        P_{t,j}
        =
        \ee^{\ii t\Theta_{V_j}}P_{0,j}\ee^{-\ii t\Theta_{V_j}}
        =
        \Pi_{tV_j}^\perp\Bigl((t-1)\frac{2\pi}{L_j}\Phi_{V_j}\Bigr).
    \end{gathered}
\]
Appealing to Lemma~\ref{charpiv}, we have that $[0,1]\ni t\mapsto P_{t,j}$ is a smooth family of pseudodifferential operators of order \(0\) on \(\Gamma\). The family $t\mapsto P_{t,j}$ differs from the Calder\'on projector by a family of compact operators and so is constant up to compact operators. Therefore, for any $t\in [0,1]$, the operator 
\[
    \slashed{D}^+_{P_t}=2\partial_{\bar z}, 
    \quad 
    \dom (\slashed{D}^+_{P_t})=\{u\in H^1(\Omega):~P_t\gamma_0(u)=0\},
\]
is a regular realization of the elliptic operator $2\partial_{\bar{z}}$. Therefore,~\cite[Theorem 8.5]{goffbandara} implies that 
\[
    \ind (\slashed{D}^+_{P_1})
    =
    \ind (\slashed{D}^+_{P_0}).
\]
To finish the proof, we notice that \(\slashed{D}^+_{P_1}=\slashed{D}_{\mathrm{APS}, V}^+\) and \(\slashed{D}^+_{P_0}= \slashed{D}_{\mathrm{APS}, 0}^+(-\tilde{\Phi}_V)\).
\end{proof}

\begin{proposition}\label{lnlnadla}
It holds that
\[
    \ind \bigl(\slashed{D}^+_{\mathrm{APS},0}(-\tilde\Phi_V)\bigr)
    = 
    \ind \bigl(\slashed{D}^+_{\mathrm{APS},0}(0)\bigr)+\sum_{j=0}^d\AC{-\Phi_{V_j}}.
\]
\end{proposition}

\begin{proof}
We use a relative index theorem (see~\cite{goffbandara,boosss, melrose} for a detailed description of the underlying theory). The boundary condition in the definition of \(\slashed{D}^+_{\mathrm{APS},0}(\vec\lambda)\) is defined by the eigenspace \(E_0(\vec\lambda)\) (see~\eqref{eq:def-op-APS*}).

\textbf{Step 1.} (Simply connected case.)

In this case \(\tilde\Phi_V=2\pi\Phi_V/L\) and we identify \(\vec\lambda\) with its single component \(\lambda\).  Recall that \(E_0(\lambda)\) is the eigenspace of \(\Dd_0=-\ii\partial_s\) corresponding to eigenvalues less than \(\lambda\), and it is explicitly known from Proposition~\ref{speckkned} (with \(V=0\)). 

If we have \(\tilde\Phi_V\leq 0\) then \(E_0(0)\subset E_0(-\tilde\Phi_V)\); consequently (see for instance~\cite[Lemma 4.5]{bandararel})
\[
    \ind (\slashed{D}^+_{\mathrm{APS},0}(-\tilde\Phi_V)
    = 
    \ind (\slashed{D}^+_{\mathrm{APS},0}(0))+\dim\bigl(E_0(-\tilde\Phi_V)/E_0(0) \bigr).
\]
Now, by Proposition~\ref{speckkned}, we have
\[
    \dim\bigl(E_0(-\tilde\Phi_V)/E_0(0) \bigr)
    =
    \#\bigl(\{m\in\Z~:~0\leq m <-\Phi_V\}\bigr)
    =
    \AC{-\Phi_V}.
\]
The case $\tilde\Phi_V>0$ is proven analogously using \(E_0(0)\supset E_0(-\tilde\Phi_V)\) and interchanging the role of $\slashed{D}^+_{\mathrm{APS},0}(-\tilde\Phi_V)$ and $\slashed{D}^+_{\mathrm{APS},0}(0)$ in the relative index theorem.

\textbf{Step 2.} (Non-simply connected case.)

By the argument in Step 1: 
\[
    \ind \bigl(\slashed{D}^+_{\mathrm{APS},0}(-\tilde\Phi_V)\bigr)
    = 
    \ind \biggl(
        \slashed{D}^+_{\mathrm{APS},0}\Bigl(
            0,-\frac{2\pi}{L_1}\Phi_{V_1},\cdots,-\frac{2\pi}{L_d}\Phi_{V_d}\Bigr)
        \biggr)
    +\AC{-\Phi_{V_0} }.
\]
Iterating this we get the formula in Proposition~\ref{lnlnadla}.
\end{proof}

\subsection{Computing the APS-index}

Now we are reduced to the computation of $\ind (\slashed{D}^+_{\mathrm{APS},0}(0))$ for a general bounded, connected $C^\infty$-domain. We can start from the case of the unit disc. 

\begin{proposition}[Disc domain]\label{discomdomd}
For the unit disc $\Omega=D(0,1)$, we have that 
\[
    \ind (\slashed{D}^+_{\mathrm{APS},0}(0)) = 0.
\]
\end{proposition}

\begin{proof}
By Fourier series,  we have for the unit disc
\[
    \begin{aligned}
        \gamma_0(\mathscr B(\Omega)) 
        & = \Bigl\{\sum_{n\geq 0} \sqrt{n+1}\, a_n\ee^{\ii ns}~:~(a_n)_{n\geq 0}\subset \ell^2(\C)\Bigr\},\\
        \gamma_0(\bar{\mathscr B}(\Omega))
        & = \Bigl\{\sum_{n\leq 0} \sqrt{|n|+1}\, a_n \ee^{\ii ns}~:~(a_n)_{n< 0}\subset \ell^2(\C)\Bigr\}.
    \end{aligned}
\]
Thus, we get from Proposition~\ref{speckkned},
\[
    \begin{aligned}
        E_0\cap\gamma_0(\mathscr B(\Omega))
        & = \overline{\Span \bigl(\{f_{m,0}~:~m\in\Z,m\geq 0,\mu_m(0,2\pi)<0\}\bigr)}=\{0\},\\
        E_{-1}^\perp\cap\gamma_0(\bar{\mathscr B}(\Omega))
        & =\overline{\Span \bigl(\{f_{m,-1}~:~m\in\Z,m \leq 0,\mu_m(-1,2\pi)\geq 0\}\bigr)} = \{0\}.
    \end{aligned}
\]
Now it follows by Proposition~\ref{lnnada}
\[
    \ind \bigl(\slashed{D}^+_{\mathrm{APS},0}(0)\bigr)
    =
    \dim\bigl(E_0\cap\gamma_0(\mathscr B(\Omega))\bigr)
    -
    \dim\bigl(E_{-1}^\perp\cap\gamma_0(\bar{\mathscr B}(\Omega))\bigr)
    =0.
    \qedhere
\]
\end{proof}

With Proposition~\ref{discomdomd} in hand we can apply a general index theorem to deal with arbitrary bounded, connected domains.

\begin{proposition}[General  connected domain]\label{lemknknknadad}
If $\Omega\subseteq \C$ is a bounded, connected domain with smooth boundary \(\Gamma\) consisting of  \(d+1\) simple closed curves,  then 
\[
    \ind(\slashed{D}^+_{\mathrm{APS},0}(0))=-d.
\]
In particular,  if \(\Omega\) is simply connected,  \(\ind(\slashed{D}^+_{\mathrm{APS},0}(0))=0\).
\end{proposition}

\begin{proof}
The direct frame \((\tau,\nu)\) gives us natural coordinates \((s,x_\n)\) valid in a collar neighborhood of the boundary, where \(s\) denotes the arc-length coordinate along \(\Gamma\) and \(x_\n\) is the normal variable,  so that \(x_n=0\) defines the boundary,  and inside the \(\Omega\) we have \(x_n>0\) (see e.g.,~\cite[Appendix~F]{FH-b})).

The operator \(\slashed{D}^+_{\mathrm{APS},0}(0)=2\partial_{\bar z}\) can be expressed in the coordinates \((s,x_\n)\) as
\[
    2\partial_{\bar z}
    = \partial_{1}+\ii \partial_{2}
    = \tau\bigl((1-x_\n\kappa(s))^{-1}\partial_s\bigr)+\nu(\partial_{x_\n})
\]
where \(\tau,\nu\) are viewed as the complex numbers \(\tau_1+\ii\tau_2,\nu=\nu_1+\ii\nu_2\) and \(\partial_1,\partial_2\) denote the partial differentiation with respect to the Cartesian coordinates. Recalling that \(\nu=\ii\tau\) by~\eqref{eq:nu=J-tau},  we eventually get\footnote{Note also that a function \(u\) in the domain \(\ind (\slashed{D}^+_{\mathrm{APS},0}(0))\)  must obey the condition \(\Pi_{\geq0}\gamma_0(u)=0\) where, in the notation of~\cite[Eq.~(1.2)]{ggrubb},  \(\Pi_{\geq 0}\) is the  orthogonal projection on the eigenspace defined by the non-negative eigenvalues of the boundary operator \(-\ii\partial_s\).}
\[
    2\partial_{\bar z}
    =
    \ii\tau \bigl(\partial_{x_\n} -\ii (1-x_\n\kappa(s))^{-1}\partial_s\bigr).
\]
So the condition in~\cite[Eq.~(1.4)]{ggrubb} holds, and we can apply Grubb's APS index formula~\cite[Theorems~1.2~\&~1.3]{ggrubb}. Since we are using the flat metric on $\C$, it follows  that 
\[
    \ind (\slashed{D}^+_{\mathrm{APS},0}(0))
    =
    \int_{\Gamma}b_0\dd s -\frac{\eta(0)+\dim\ker (-\ii\partial_s)}{2},
\]
where $b_0$ is a local term defined from universal invariant polynomial in the second fundamental form and its covariant derivatives along the normal to the boundary, and \(\eta(0)\) is the \(\eta\) invariant of  \(\Dd_0=-\ii\partial_s\), defined as follows. For \(\mathrm{Re}(a)>1\), we let
\[
    \eta(a) = \sum_{\lambda_m\not=0}\mathrm{sgn}(\lambda_m)|\lambda_m|^{-a},
\]
and extend by meromorphicity to the complex plane. Here \(\lambda_m\) are the  eigenvalues of \(\Dd_0\) which can be retrieved from Proposition~\ref{speckkned} by observing that \(\Dd_0=\dsum \Dd_0|_{\Gamma_j}\),   where \((\Gamma_j)_{j=0}^d\) are the connected components of \(\Gamma\). Since the spectrum of each $\Dd_0|_{\Gamma_j}$ is symmetric around the origin, we find \(\eta(a)=0\) for \(\mathrm{Re}(a)>1\) by evaluating the sum defining \(\eta(s)\) on each connected component separately. In particular, $\eta(0)=0$.
Moreover, we have (in \(L^2(\Gamma)\cong\dsum L^2(\Gamma_j)\))
\[
    \dim\ker (-\ii\partial_s)=1+d.
\]
Finally, from scaling considerations, it follows that there is a universal real constant $c$ such that $b_{0}=c\,\kappa$, where $\kappa$ denotes the curvature of $\Gamma$. The universality of $c$ means that it is independent of the domain. By the Gauss--Bonnet theorem, we have
\[
    \int_\Gamma\kappa(s)\dd s
    =
    \int_{\Gamma_0}\kappa(s)\dd s
    + 
    \sum_{j=1}^d\int_{\Gamma_j}\kappa(s)\dd s
    =
    2\pi -2\pi d,
\]
where \(\Gamma_0\) is the outer boundary of the domain \(\Omega\); the negative term is due to the positive orientation in the interior boundary where \(\nu\) points inward \(\Omega\) (so outward the domain enclosed by the curve  \(\Gamma_j\)).

Therefore, summing all contributions, we obtain
\[
    \ind (\slashed{D}^+_{\mathrm{APS},0}(0))
    =
    2\pi(1-d)c - \frac{1+d}{2}.
\]
Universality of the constant $c$ implies, after comparison with the case of the unit disc, Proposition~\ref{discomdomd}, that \(c=1/4\pi\). Substituting into the foregoing index formula finishes the proof.
\end{proof}

Now we conclude the paper by finishing the proof of our main theorem.

\begin{proof}[Proof of Theorem~\ref{thm:main}]
	As we mentioned earlier, we assume that the magnetic potential  \(\Ab\) is given by~\eqref{eq:def-A} and refer to Remark~\ref{rem:generalmagn} below for the general case.
	
By~\eqref{eq:N->N-O} and Remark~\ref{rem:index},  we have that
\[
    N(\Hb_{\Ab,g},0)
    \geq
    \ind (\slashed{D}_{\phi,\mathrm{APS}, V}^+)
\]
with \(V=g-A_\tau\) and \(A_\tau\) introduced in~\eqref{eq:def-a}.

We next collect~\eqref{eq:index-phi=0} and the conclusions in Propositions~\ref{lemkmknknk} and~\ref{lnlnadla}, to obtain
\[ 
    \ind (\slashed{D}_{\phi,\mathrm{APS}, V}^+)
    =
    \ind (\slashed{D}^+_{\mathrm{APS},0}(0))+\sum_{j=0}^d\AC{-\Phi_{V_j}},
\]
where \(\ind (\slashed{D}^+_{\mathrm{APS},0}(0))=-d\) by Proposition~\ref{lemknknknadad}.  Thus, remembering the definition of \(V\) above, we get that
\[
    N(\Hb_{\Ab,g},0)
    \geq -d
    +\sum_{j=0}^d\AC{-\Phi_{V_j}}=-d+ \sum_{j=0}^d \AC{\Phi_j-\Phi_{g,j}}.\qedhere
\]
\end{proof}

\begin{remark}\label{rem:lpb}
The proof  above extends to the case $B\in L^p(\Omega,\R)$ for $p>2$. Elliptic regularity implies that $\phi\in W^{2,p}(\Omega,\R)$, and so $\Ab\in W^{1,p}(\Omega,\R^2)$. The trace theorem implies that $A_\tau:=\tau\cdot \Ab|_\Gamma\in W^{1-\frac{1}{p},p}(\Gamma,\R)$ and by the Sobolev embedding theorem, $A_\tau\in C(\Gamma,\R)$ if $p>2$. 

The results in Section~\ref{sec:useful},~\ref{sec:Dirac}, and~\ref{sec:eginnad} carry through ad verbatim for $B\in L^p(\Omega,\R)$, $p>2$, and $A_\tau\in C(\Gamma,\R)$. Section~\ref{sec:index} is a bit more subtle, where Proposition~\ref{lnnada} and~\ref{lemkmknknk} uses some regularity. Their proofs extend since $\ee^{i\Theta_{A_\tau}}\in C^1(\Gamma,\R)$ and it is sufficient that $\ee^{i\Theta_{A_\tau}}\in C^{1/2+\epsilon}(\Gamma,\R)$ for $[\Pi_0(0),\ee^{i\Theta_{A_\tau}}]$ to be compact on $H^{s}(\Gamma)$ for $s\in (-1/2-\epsilon,1/2+\epsilon)$~\cite[Theorem 1.4]{gimpgoff}, and in particular \(\Pi^{\perp}_V\) differs from the Calder\'on projector by an operator compact on $H^{1/2}(\Gamma)$ and $H^{-1/2}(\Gamma)$ by Lemma~\ref{charpiv}.
\end{remark}

\begin{remark}\label{rem:lowome}
Extending the index theoretical proof of Theorem~\ref{thm:main} to the case when $\Omega$ is non-smooth is harder. This is due to the fact that Atiyah--Patodi--Singer index theory has only been developed in sufficiently large degree of regularity of $\partial\Omega$. We note that Lemma~\ref{eq:confinv} below can be combined with the discussion above to show that Theorem~\ref{thm:main} holds when $\Omega$ is simply connected with Dini-smooth boundary and $B\in L^p(\Omega)$ for $p>2$. 
\end{remark}

\begin{remark}\label{rem:generalmagn}
As discussed in Remark~\ref{generapotential}, the proof of Theorem~\ref{thm:main} above extends to general $\Ab\in C^\infty(\overline{\Omega},\R^2)$. The generalization depends on replacing $\ee^{-\phi}\partial_{\bar{z}}\ee^\phi$ with $\slashed{D}_{\Ab}^+:=\partial_{\bar{z}}-i\overline{\Ab}$, which since it does not change the principal symbol affects little in the general theory of first order elliptic differential operators we are invoking. Indeed, the proof of Proposition~\ref{prop:identity} extends ad verbatim when replacing $\ee^{-\phi}\partial_{\bar{z}}\ee^\phi$ with $\slashed{D}_{\Ab}^+$. If we further replace the Bergman space $\mathscr B(\Omega)$ with the $L^2$-kernel of $\slashed{D}_{\Ab}^+$ in the discussions in Subsections~\ref{subsec:newaux},~\ref{subsec:gendom} and~\ref{subsec:intermezzo}, analogous statements as in~\eqref{eq:N->N-O} and Remark~\ref{rem:index} lead to the lower bound
\[
    N(\Hb_{\Ab,g},0)
    \geq
    \ind (\slashed{D}_{\Ab,\mathrm{APS}, V}^+)
\]
Here $\slashed{D}_{\Ab,\mathrm{APS}, V}^+$ denotes the differential operator $\slashed{D}_{\Ab}^+$ equipped with the APS-boundary condition~\eqref{eq:def-dom-APS} defined from $E_V$ (see~\eqref{eq:def-E-V}) where $V=g-\tau\cdot \Ab$. To prove Theorem~\ref{thm:main} for general $\Ab$, it therefore suffices to have the index formula
\begin{equation}\label{indexforgena}
    \ind (\slashed{D}_{\Ab,\mathrm{APS}, V}^+)
    =
    -d+\sum_{j=0}^d\AC{\Phi_{j}-\Phi_{g,j}}.
\end{equation}
Using homotopy invariance and relative index theory as in Subsection~\ref{subsec:relhomo}, we can without loss of generality assume $g=0$.

While less conceptual than the proof above for the Dirichlet gauge, the index formula~\eqref{indexforgena} follows from Grubb's APS-index formula~\cite{ggrubb}. The latter implies that
\begin{equation}\label{grubbsaps}
    \ind (\slashed{D}_{\Ab,\mathrm{APS}, V}^+)
    =
    \frac{1}{2\pi} \int_\Omega B\dd x+\int_{\Gamma}b_0\dd s(x) -\frac{\eta_{\Dd_V}(0)+\dim\ker \Dd_V}{2},
\end{equation}
where the $\eta$-invariant $\eta_{\Dd_V}(0)$ is the value at $s=0$ of the meromorphic function $\eta_{\Dd_V}$, defined for $\re s>1$ by
\[
    \eta_{\Dd_V}(s):=\sum_{\lambda\in \mathrm{Spec}(\Dd_V)\setminus \{0\}} \mathrm{sign}(\lambda)|\lambda|^{-s},
\]
with eigenvalues counted by multiplicity. Moreover, $b_0$ is a local term defined from a universal invariant polynomial in the second fundamental form and curvatures. By well known techniques from invariance theory~\cite{gilkey}, degree reasons and the low dimensionality we know that $b_0=c\kappa+c'\partial_{x_\n}(\dd x_\n\neg (B\dd x_1\wedge \dd x_2))$, where $\kappa$ denotes the curvature of $\Gamma$ and $c,c'$ are universal constants. The value $c=1/4\pi$ is deduced from Lemma~\ref{lemknknknadad} and $c'=0$ can be computed from Proposition~\ref{prop:disc-cst}. In particular, we have $\int_{\Gamma}b_0\dd s(x)=(1-d)/2$ by the Gauss--Bonnet theorem. A computation with Hurwitz zeta functions~\cite[p.\ 264, p.\ 268]{Ap} and Proposition~\ref{speckkned} implies that 
\[
    \frac{\eta_{\Dd_V}(0)+\dim\ker \Dd_V}{2}=\frac{1+d}{2}+\sum_{j=0}^d \left( \Phi_j-\AC{\Phi_{j}} \right).
\]
Summing up the terms in~\eqref{grubbsaps} and using $\Phi=\sum_{j=0}^d \Phi_j$ we conclude~\eqref{indexforgena}.
\end{remark}

\section{ Proof of Theorem~\ref{thm:main} using a trace identity}\label{sec:traceproof}

In this section we provide an alternative proof of the first part of Theorem~\ref{thm:main} in the simply connected case.   It is based on the identity from Proposition~\ref{prop:identity}, but instead of index theory we use a trace identity for the Benjamin--Ono equation. In this approach the simple connectivity of $\Omega$ is used and is assumed throughout this section.

\subsection{Reduction to the unit disc}

By the Riemann mapping theorem, there is a biholomorphic $F:D(0,1)\to\Omega$.  When $\Omega$ has sufficiently smooth boundary, $F$ together with its derivatives extends continuously to $\overline{D(0,1)}$, and this extension induces a $C^1$-diffeomorphism $f:\partial D(0,1)\to\Gamma = \partial\Omega$ (see~\cite[Thm. 2.1]{Po}).  Define
\begin{equation}\label{eq:tildeonb}
    \tilde B \coloneqq |F'|^2 \ (B\circ F) : D(0,1) \to\R \,,
    \quad
    \tilde g \coloneqq |f'| \ (g\circ f) : \partial D(0,1)\to\R \,.
\end{equation}
Then, by a change of variables,
\[
    \frac1{2\pi} \int_\Omega B(x)\dd x = \frac1{2\pi} \int_{D(0,1)} \tilde B(x)\dd x \,,
    \quad
    \frac1{2\pi} \int_\Gamma g\dd s = \frac1{2\pi} \int_{\partial D(0,1)} \tilde g\dd s \,. 
\]
The following lemma uses these identities to show that the number of negative eigenvalues does not change when passing from the triple $(\Omega,B,g)$ to the triple $(D(0,1),\tilde B,\tilde g)$. 

\begin{lemma}\label{eq:confinv}
If $\Omega\subseteq \R^2$ is bounded,   simply connected domain bounded by a simple closed Dini-smooth Jordan curve, then we have 
\[
    N(\Hb_{\Ab,g},0) = N(\mathscr{H}^{D(0,1)}_{\tilde \Ab,\tilde g},0),
\]
where $\tilde B$ and $\tilde g$ are defined as in~\eqref{eq:tildeonb} from the Riemann mapping theorem.
\end{lemma}

\begin{proof}
First, for the continuous extension of the Riemann mapping $F$ to a map from $\overline{D(0,1)}$ to $\overline\Omega$, it suffices that $\Omega$ is bounded by a simple closed Jordan curve, and in this case the extension of $F$ maps $\partial D(0,1)$ homeomorphically to $\Gamma=\partial\Omega$. Next, by~\cite[Theorem 3.5]{Po}, if $\Omega$ is bounded by a simple closed Dini-smooth Jordan curve, then $F'$ extends continuously to the boundary.

We set
\[
    \tilde \Ab \coloneqq (DF)^{\mathrm{T}} (\Ab \circ F) \,,
\]
where $DF$ is the $2\times 2$ Jacobi matrix of $F$. A computation based on the Cauchy--Riemann equations for $F$ shows
\[
    \curl\tilde \Ab = \tilde B \,.
\]
Given $\psi\in H^1(\Omega)$, let
\[
    \tilde\psi \coloneqq \psi \circ F \,.
\]
Then a simple computation shows that $\tilde\psi\in H^1(D(0,1))$ and $(-i\nabla-\tilde \Ab)\tilde\psi = (DF)^{\mathrm{T}} (((-i\nabla -\Ab)\psi)\circ F)$, and so, by the Cauchy--Riemann equations,
\[
    |(-i\nabla-\tilde \Ab)\tilde\psi|^2 
    = 
    |F'|^2 (|(-i\nabla -\Ab)\psi|^2 \circ F).
\]
Thus, by a change of variables,
\[
    \int_\Omega |(-i\nabla-\Ab)\psi|^2\dd x 
    = 
    \int_{D(0,1)} |(-i\nabla -\tilde \Ab)\tilde\psi|^2\dd x \,.
\]
Similarly, $\tilde B|\tilde\psi|^2 = |F'|^2 ((B|\psi|^2)\circ F)$ implies
\[
    \int_\Omega B |\psi|^2\dd x 
    = 
    \int_{D(0,1)} \tilde B |\tilde\psi|^2\dd x \,.
\]
Applying Glazman's lemma~\eqref{eq:glazman} twice, once with $T=\Hb_{\Ab,g}$ and once with $T=\mathscr{H}^{D(0,1)}_{\tilde B,\tilde g}$, we obtain the claimed equality.
\end{proof}

\begin{remark}
Further conditions in terms of H\"older regularity of $F$ are given by the Kellogg--Warschawski theorem~\cite[Theorem 3.6]{Po}.
\end{remark}

\subsection{Analysis on the boundary of the unit disc}

The conclusion of Lemma~\ref{eq:confinv} is that it suffices to prove the first part of Theorem~\ref{thm:main} for the unit disc. Thus, from now on we assume that $\Omega=D(0,1)$ and we write $B$ and $g$ instead of $\tilde B$ and $\tilde g$.

We denote by $P_+$ the projection in $L^2(\partial D(0,1))$ onto the Hardy space, that is, onto the space spanned by the functions $e^{ims}$ with $m\geq 0$. We consider the operator
\[
    P_+\Dd_V P_+ 
    = 
    P_+ (-\ii\partial_s+V) P_+
\]
in $L^2(\partial D(0,1))$ with $V= g- A_\tau$. We claim that
\begin{equation}\label{eq:redbdry}
	N(\Hb_{\Ab,g},0) 
    \geq 
    N(P_+\Dd_V P_+,0) \,.
\end{equation}
This is the statement of Proposition~\ref{knlknad*},  but we briefly recall the argument in the current setting.  The inequality follows by another application of Glazman's lemma~\eqref{eq:glazman}. Indeed, we denote by $E$ the extension operator from ${\im P_+}$ to $\mathscr B(D(0,1))$, defined by $Ee^{im\cdot}(z) = z^m$ for $m\geq 0$. Then for $v\in\im(\mathbf 1_{(-\infty, 0)}(P_+\Dd_V P_+))$, we have\footnote{The operator \(P_+ \Dd_V P_+\) is bounded from below,    so \(v\in \im(\mathbf 1_{(-c, 0)}(P_+\Dd_V P_+))\subset H^1(\Gamma)\), where  \(c\) is some positive constant.} $Ev\in H^1(D(0,1))\cap \mathscr B(D(0,1))$ and, by Proposition~\ref{prop:identity},
\[
q_{\Ab,g}(Ev) = \IP{ P_+ \Dd_V P_+ v,v}_{L^2(\partial D(0,1))} \leq 0 \,.
\]
The inequality is strict if $v\neq 0$. Therefore, the assertion~\eqref{eq:redbdry} follows from Glazman's lemma~\eqref{eq:glazman} with $T=\Hb_{\Ab,g}$ and $\mathcal M = \im( E \mathbf 1_{(-\infty, 0)}(P_+\Dd_V P_+))$, together with the fact that $E$ is injective on $\im P_+$. The proof of the first part of Theorem~\ref{thm:main} is therefore complete, once we have shown the following lemma.

\begin{lemma}
	$N(P_+\Dd_V P_+,0) \geq \AC{-\Phi_V }$.
\end{lemma}

We note that the lemma is obvious if $V$ is a constant, because then the quadratic form of $P_+\Dd_V P_+$ is negative on the subspace spanned by $\ee^{\ii ms}$ with $0\leq m < -V=-\Phi_V$, which has dimension $\AC{-\Phi_V}$.

\begin{proof}
We first prove a variant of the lemma on the line. We denote by $\Pi_+$ the orthogonal projection in $L^2(\R)$ onto the  Hardy space  of functions whose Fourier transform vanishes on $(-\infty,0)$. We consider the operator $\Pi_+(-\ii\partial+W)\Pi_+$ in $L^2(\R)$. This is an operator that appears in the Lax pair for the Benjamin--Ono equation. A trace formula for this equation reads
\[
    N(\Pi_+(-\ii\partial+W)\Pi_+,0) - \frac1{(2\pi)^2} \int_0^\infty |\beta(\lambda)|^2\, \frac{\dd\lambda}{\lambda} 
    = 
    - \frac{1}{2\pi} \int_{\R} W(t)\dd t \,;
\]
see, e.g., \cite[Equation (113)]{KM}. Here $\beta$ is a certain scattering coefficient whose definition is irrelevant for us. From this trace formula, it follows immediately that $N(\Pi_+(-\ii\partial+W)\Pi_+,0) \geq - (2\pi)^{-1} \int_{\R} W(t)\,dt$.

Now we transfer the result from the line to the circle.  We first note that, if $W$ on $\R$ and $V$ on $\partial D(0,1)$ are related by
\[
    W(t) 
    = 
    \frac{2}{1+t^2}\ V\left(\frac{i-t}{i+t}\right),
\]
then
\[
\int_{\R} W(t)\dd t = \int_{\partial {D(0,1)}} V(s)\dd s \,.
\]
(Here we slightly abuse notation and do not distinguish between $V(s)$ and $V(e^{\ii s})$.) This follows by a standard change of variables. Further, we claim that the number of negative eigenvalues of $P_+(-\ii\partial_s +V)P_+$ coincides with the number of negative eigenvalues of $\Pi_+(-\ii\partial_t+W)\Pi_+$. This follows by Glazman's lemma~\eqref{eq:glazman}, similarly as in the previous subsection, provided one notes that if $\psi\in L^2(\partial {D(0,1)})$ and $\Psi\in L^2(\R)$ are related by
\[
    \Psi(t) = \psi\left(\frac{i-t}{i+t}\right) \,,
\]
then $\psi\in\im P_+$ if and only if $\Psi\in\im \Pi_+$, and
\[
    \int_{\partial {D(0,1)}} \overline\psi (-\ii\partial_s)\psi\dd s 
    =
    \int_{\R} \overline\Psi(-\ii\partial_t)\Psi\dd t \,,
    \quad 
    \int_{\partial {D(0,1)}} W |\psi|^2\,d\theta 
    =
    \int_{\R} V |\Psi|^2\,dt \,.
\]
This concludes the proof of the lemma.
\end{proof}

\begin{remark}
The trace formulas for the Benjamin-Ono equation have been extensively studied and are widely recognized in the field. Furthermore, by incorporating the methods presented in~\cite{Wu}, it is possible to provide additional details and strengthen the proof presented in~\cite{KM} and clarify the regularity required on $V$.  For further details on the Benjamin-Ono equation and its trace formulas see~\cite{CW,AT}.
\end{remark}

\begin{remark}
It would have been more efficient, but probably less intuitive, if we would have mapped $\Omega$ conformally onto the upper half-plane $\C_+$. Then we could have directly appealed to the result for the Benjamin--Ono equation on $\R$, without doing another conformal transformation.
\end{remark}

\section*{Acknowledgement}

The authors take the opportunity to thank the Knuth and Alice foundation (grant KAW 2021.0259) for the possibility to host A.~Kachmar in Lund for six months. The first listed author was supported by the grant 0135-00166B from the Independent Research Fund Denmark. The second listed author was partially supported by the grant DMS-1954995 from the US National Science Foundation and grant EXC-2111-390814868 from the German Research Foundation. The third listed author was supported by the Swedish Research Council Grant VR 2018-0350.  SF and AK acknowledge CAMS-AUB where part of this work was carried out.

\bibliographystyle{abbrv}
\bibliography{FGKS}

\end{document}